\documentclass[a4paper]{amsart}
\pdfoutput=1
\usepackage{amsmath, amsthm}
\usepackage{amssymb, latexsym}
\usepackage{amsfonts}
\usepackage{hyperref}

\usepackage{pinlabel}

\usepackage{graphicx,color}

\usepackage[all]{xy}

\usepackage{subfigure}

\usepackage{mathdots}

\renewcommand{\labelenumi}{(\arabic{enumi})}
\renewcommand{\labelenumii}{(\alph{enumii})}
\newcommand{\enumia}{\renewcommand{\labelenumi}{(\arabic{enumi})}}
\newcommand{\enumir}{\renewcommand{\labelenumi}{(\roman{enumi})}}
\newcommand{\enumiir}{\renewcommand{\labelenumii}{(\roman{enumii})}}

\theoremstyle{plain}

\newtheorem*{claim*}{Claim}
\newtheorem{thm}{Theorem}[section]
\newtheorem{claim}{Claim}
\newtheorem{claim2}{Claim}
\newtheorem{conj}[thm]{Conjecture}
\newtheorem{cor}[thm]{Corollary}

\newtheorem{lem}[thm]{Lemma}

\newtheorem{prop}[thm]{Proposition}

\newtheorem*{thm312}{Theorem 3.12}
\newtheorem*{cor313}{Corollay 3.13}
\newtheorem*{thm51}{Theorem 5.1}
\newtheorem*{cor52}{Corollay 5.2}
\newtheorem*{prop64}{Proposition 6.4}

\theoremstyle{definition}
\newtheorem{notation}[thm]{Notation}

\newtheorem{defn}[thm]{Definition}

\theoremstyle{remark}
\newtheorem{rem}[thm]{Remark}
\newtheorem{exmp}[thm]{Example}

\newcommand{\ch}{\mathrm{ch}}
\newcommand{\cob}{\overline{\mathrm{CO}}}

\def\co{\colon\thinspace}

\font\msamfont=msam10
\def\qed{\hbox{~~\msamfont\char'003}}

\def\a1{{a^{-1}}}
\def\b1{{b^{-1}}}
\def\c1{{c^{-1}}}
\def\d1{{d^{-1}}}
\def\e1{{e^{-1}}}
\def\f1{{f^{-1}}}

\newcommand{\bbb}[1]{\ensuremath{\mathbb{#1}}}

\newcommand{\Z}{\bbb{Z}}

\newcommand{\script}[1]{\ensuremath{\mathcal{#1}}}
\def\AAA{\script{A}}
\def\BBB{\script{B}}
\def\CCC{\script{C}}

\def\FFF{\script{F}}

\def\HH{\script{H}}
\def\KK{\script{K}}
\def\NNN{\script{N}}

\newcommand{\smallcaps}[1]{\textrm{\textsc{#1}}}

\newcommand{\join}{\smallcaps{Join}}

\newcommand{\link}{\smallcaps{link}}

\newcommand{\ba}{\begin{array}}
\newcommand{\ea}{\end{array}}
\newcommand{\bc}{\begin{center}}
\newcommand{\ec}{\end{center}}
\newcommand{\be}{\begin{enumerate}}
\newcommand{\ee}{\end{enumerate}}
\newcommand{\bd}{\begin{defn}}
\newcommand{\ed}{\end{defn}}
\newcommand{\bi}{\begin{itemize}}
\newcommand{\ei}{\end{itemize}}
\newcommand{\bs}{\begin{slide}}
\newcommand{\es}{\vfil\end{slide}}
\newcommand{\bt}{\begin{tabular}}
\newcommand{\et}{\end{tabular}}

\begin{document}

\title{On  Right-Angled Artin Groups Without Surface Subgroups}
\author{Sang-hyun Kim}
\address{Department of Mathematics, the University of Texas at Austin, Austin, TX 78712-0257 \\ USA}
\email{shkim@math.utexas.edu}
\date{\today}



\begin{abstract}    
We study the class $\NNN$ of graphs, the right-angled Artin groups defined 
on which do not contain surface subgroups. 
We prove that a presumably smaller class $\NNN'$ is closed under amalgamating along complete subgraphs,
and also under adding bisimplicial edges. It follows that chordal graphs and
chordal bipartite graphs belong to $\NNN'$.
\end{abstract}

\maketitle



\section{Introduction}\label{sec:intro}

In this paper, a {\em graph} will mean a finite graph without loops or multi-edges. For a graph $\Gamma$, let $V(\Gamma)$ and $E(\Gamma)$ denote the vertex set and the edge set of $\Gamma$, respectively.
The {\em right-angled Artin group with the underlying graph
$\Gamma$} is the group presentation
\[
A(\Gamma) = \langle V(\Gamma) \; | \; [u,v]=1\mbox{ for }\{u,v\}\in E(\Gamma)\rangle
.\]
Also known as {\em graph groups} or {\em partially commutative groups}, right-angled Artin groups  possess various group theoretic properties. 
One of the most fundamental results is that, two right-angled Artin groups are isomorphic if and only if their underlying graphs are isomorphic~\cite{KMNR1980,droms1987}. 
Right-angled Artin groups are  linear~\cite{humphries1994,HW1999,DJ2000}, biorderable~\cite{DT1992, CW2004}, 
and acting on finite-dimensional CAT(0) cube complexes freely and cocompactly~\cite{CD1995,MV1995,NR1998}. Any subgroup of a right-angled Artin group
surjects onto $\Z$~\cite{howie1982}.

Certain group theoretic properties of $A(\Gamma)$ can be 
detected by graph theoretic properties of $\Gamma$.
The {\em complement graph}  of  a graph $\Gamma$ is the graph $\overline\Gamma$,
defined by
$V(\overline\Gamma)=V(\Gamma)$ and $E(\overline\Gamma)=\{ \{u,v\} \; | \; \{u,v\}\not\in E(\Gamma)\}$.
For a subset $S$ of $V(\Gamma)$, the {\em induced subgraph $\Gamma_S$ of $\Gamma$ on $S$}
is the maximal subgraph of $\Gamma$ with the vertex set $S$. We also write $\Gamma_S\le\Gamma$.
Note that
$V(\Gamma_S)=S$ and 
$E(\Gamma_S)=\{\{u,v\}\;|\;u,v\in S\textrm{ and }\{u,v\}\in E(\Gamma)\}$.
If $\Lambda$ is another graph, 
an {\em induced} $\Lambda$ in $\Gamma$  is
an induced subgraph of $\Gamma$, which is isomorphic to $\Lambda$. 
An elementary fact is, if $\Gamma$ contains an induced $\Lambda$,
then $A(\Lambda)$ embeds into $A(\Gamma)$.
A {\em complete graph} $K_n$ is a graph with $n$ vertices 
such that every pair of distinct vertices are joined by an edge.
For convention, $K_0=\varnothing$ is considered also as a complete graph.
$C_n$ and $P_n$ denote the cycle and the path with $n$ vertices, respectively.
In particular, $P_n$ is obtained by removing an edge in $C_n$.
$C_3$ is also called a {\em triangle}.
A graph is  {\em chordal} if the graph does not contain any induced cycle of length at least 4.
Graph theoretic characterizations of $\Gamma$
determine several group theoretic
properties of $A(\Gamma)$:
\begin{itemize}
\item
$A(\Gamma)$ is 
coherent (i.e.~every finitely generated subgroup is finitely presented), if and only if 
$\Gamma$ is {\em chordal}~\cite{droms1987a},
if and only if
 $A(\Gamma)$ has a free commutator subgroup~\cite{SDS1989}.
\item
$A(\Gamma)$ is 
subgroup separable (i.e.~every finitely generated subgroup is closed in the profinite topology),
 if and only if $\Gamma$ does not contain an induced $C_4$ or an induced $P_4$~\cite{MR2006},
 which happens exactly when
  every subgroup of $A(\Gamma)$ is also a right-angled
Artin group~\cite{droms1987b}.
\item
$A(\Gamma)$ is 
virtually a 3--manifold group, if and only if each connected component 
of $\Gamma$ is a tree or  a triangle~\cite{droms1987a,gordon2004}. 
\end{itemize} 

In this paper, a {\em surface} will mean a compact, oriented $2$--manifold.
A {\em closed} ({\em compact}, respectively) {\em hyperbolic surface group} will mean the fundamental group of 
a closed (compact, respectively) hyperbolic surface.
Finding sufficient and necessary conditions for a given group 
to contain a closed hyperbolic surface group
is an important question motivated by 3--manifold theory. 
In this article, we consider this question in the case of right-angled Artin groups.
Namely, we investigate
 \[ \NNN = \{ \Gamma \; | \; A(\Gamma)\mbox{ does not contain a closed hyperbolic surface group} \} .\]
 
$A(\Gamma)$ is known to contain a closed hyperbolic surface group if
$\Gamma$ has an induced $C_n$ ($n\ge5$)~\cite{SDS1989} 
or an induced $\overline{C_n}$ ($n\ge5$)~\cite{kim2008} (proved by the author). 
That is, $C_n$ and $\overline{C_n}$ are not in $\NNN$ for $n\ge5$. The  classification of
 the graphs in $\NNN$ with at most 8 vertices is given in~\cite{CSS2008}. 
 A {\em complete graph amalgamation} of two graphs is the union of the two graphs such that their intersection is complete;
in particular, those two graphs will be induced subgraphs of the union.
 A key stumbling block for the 
 (graph theoretic) characterization of $\NNN$ is the following conjecture.
 
 \begin{conj}\label{conj:nclosedkn}
 $\NNN$ is closed under complete graph amalgamation.
 \end{conj}
 
For a graph $\Gamma$,  
a cube complex $X_\Gamma$ is inductively defined as follows.
\enumir
\be
\item
$X_\Gamma^{(0)}$ is a single vertex.
\item
Suppose $X_\Gamma^{(k-1)}$ is constructed, so that each complete subgraph of $\Gamma$
with $i$ vertices $(i<k)$ corresponds to an $i$--torus.
Let  $K$ be a complete subgraph of $\Gamma$ with $k$ vertices.
Glue a unit $k$--cube to $X_\Gamma^{(k-1)}$ so that distinct pairs of parallel faces are glued to
distinct $(k-1)$--tori corresponding to the complete subgraphs of $K$ with $k-1$ vertices.
\ee
$X_\Gamma$, called
the {\em Salvetti complex of $A(\Gamma)$}, is 
 a locally CAT(0) cube complex, on which $A(\Gamma)$ acts freely and cocompactly.
In particular, $X_\Gamma$ is a $K(A(\Gamma),1)$--space~\cite{charney2007}.

Suppose $\Gamma = \Gamma_1\cup \Gamma_2$, such that $K = \Gamma_1\cap\Gamma_2$
 is complete and $\Gamma\not\in\NNN$.
One can find
 a closed hyperbolic surface $S$, and 
 a $\pi_1$--injective map $f\co S\rightarrow X_\Gamma$. Since $A(\Gamma)$ is 
the amalgamated free product of $A(\Gamma_1)$ and $A(\Gamma_2)$
   along a free abelian subgroup $A(K)$, a transversality argument shows that there exists a compact
   hyperbolic surface $S_1\subseteq S$ and a $\pi_1$--injective map $g \co S_i\rightarrow X_{\Gamma_i}$ such that 
   $g(\partial S_i)\subseteq X_K$, for $i=1$ or $2$.  
Hence in order to approach
Conjecture~\ref{conj:nclosedkn},
  it is  natural to consider
a {\em relative embedding} of a compact hyperbolic surface group into $A(\Gamma)$, as follows.

\begin{defn}\label{defn:relative}
Let $\Gamma$ be a graph and $S$ be a compact hyperbolic surface. An embedding $\phi\co\pi_1(S)\rightarrow A(\Gamma)$
is called a {\em relative embedding} if
$\phi=f_*$ for some $\pi_1$--injective map
$f\co S\rightarrow X_\Gamma$  satisfying the following:
\begin{quote}
for each boundary component $\partial_i S$ of $S$, there exists a complete subgraph $K\le\Gamma$ such that
$f(\partial_i S) \subseteq X_K$.
\end{quote}
\end{defn}
\enumia

Define $\NNN'$ to be the class of graphs $\Gamma$ such that
$A(\Gamma)$ does not allow a relative embedding of a compact hyperbolic surface group. 
It is vacuously true that
$\NNN'\subseteq\NNN$. Also,
the paragraph preceding Definition~\ref{defn:relative} has proved the following.

\begin{lem}\label{lem:nni}
If $\Gamma\not\in\NNN$ and $\Gamma$ is a complete graph amalgamation of $\Gamma_1$
and $\Gamma_2$, then $\Gamma_i\not\in\NNN'$ for $i=1$ or $2$.\qed
\end{lem}

For a compact surface $S$ and a set $V$, a {\em $V$--dissection on $S$} is a pair $(\HH,\lambda)$ such that
$\HH$ is a set of transversely oriented simple closed curves and properly embedded arcs on $S$, and
$\lambda\co\HH\rightarrow V$~\cite{CW2004}. 
For each $\gamma\in\HH$, $\lambda(\gamma)$ is called the {\em label} of $\gamma$.
Note that our definition allows curves or arcs of the same label to intersect, while the definition in~\cite{CW2004} does not.
Let $\Gamma$ be a graph and $(\HH,\lambda)$ be
a $V(\Gamma)$-dissection.
Suppose that for any $\alpha$ and $\beta$ in $\HH$, $\alpha\cap\beta\ne\varnothing$ only if $\lambda(\alpha)$ and $\lambda(\beta)$
are equal or adjacent in $\Gamma$. Then $(\HH,\lambda)$
 determines a map $\phi\co\pi_1(S)\rightarrow A(\Gamma)$. $\phi$ maps
the equivalence class of a based, oriented loop $\alpha\subseteq S$  onto the word in $A(\Gamma)$, obtained by
reading off the labels of the curves and the arcs in $\HH$ that $\alpha$ intersects, and recording the order  and 
the transverse orientations of the intersections. 
That is, when $\alpha$ crosses $\gamma\in\HH$, one records $\lambda(\gamma)$ or $\lambda(\gamma)^{-1}$,
according to whether the orientation of $\alpha$ matches the transverse orientation of $\gamma$. This map 
 $\phi\co\pi_1(S)\rightarrow A(\Gamma)$
 is called the {\em label-reading map with respect to } (or, {\em induced by}) $(\HH,\lambda)$, and 
$(\HH,\lambda)$ is called a {\em label-reading pair with the underlying graph $\Gamma$}. In~\cite{CW2004}, it is shown that any map $\pi_1(S)\rightarrow A(\Gamma)$ can
be realized as a label-reading map. By studying this label-reading pair, we will prove that a relative embedding of a compact hyperbolic surface group
into $A(\Gamma)$ can be { promoted} to an embedding of a closed hyperbolic surface group into $A(\Gamma^*)$, for some graph $\Gamma^*$ which is strictly larger than $\Gamma$ (Lemma~\ref{lem:promotion}). This plays a crucial role in the proof of the following theorem.

\begin{thm312}\label{thm:312}
$\NNN'$ is closed under complete graph amalgamation.
\end{thm312}

$K_n\in\NNN'$, since $A(K_n)\cong \Z^n$ does not contain any non-abelian free group. A classical result of Dirac shows that
any chordal graph can be constructed by taking complete graph amalgamations successively, starting from complete graphs~\cite{dirac1961,golumbic2004}. So, we have:

\begin{cor313}\label{cor:313}
All chordal graphs  are in $\NNN'$.
\end{cor313}

Given a label-reading pair $(\HH,\lambda)$ inducing $\phi \co \pi_1(S)\rightarrow A(\Gamma)$, we will define 
the {\em complexity} of $(\HH, \lambda)$.
A label-reading pair $(\HH,\lambda)$  is called {\em normalized} if 
  the complexity of $(\HH,\lambda)$ is minimal in the lexicographical ordering, 
  among all the label-reading pairs inducing the same map up to conjugation in $A(\Gamma)$.
   Certain properties of a label-reading map can be more easily detected by looking at this normalized label-reading pairs.
 An edge $\{a,b\}$ of a graph $\Gamma$ is called {\em bisimplicial}
 if any vertex adjacent to $a$ is either equal or adjacent to
any vertex that is adjacent to $b$~\cite{GG1978}. For $e\in E(\Gamma)$, $\mathring{e}$ denote the interior of $e$.
 
 \begin{thm51}\label{thm:51}
 Let $e$ be a bisimplicial edge of $\Gamma$. If  $\Gamma\setminus\mathring{e}\in\NNN'$, then $\Gamma\in\NNN'$.
 \end{thm51}

A {\em chordal bipartite} graph is a graph that does not contain a triangle or an induced cycle of length at least 5.
By definition, a chordal bipartite graph is not necessarily chordal.
Any chordal bipartite graph can be obtained by successively attaching bisimplicial edges, starting from 
a discrete graph~\cite{GG1978}. Using this, we prove:
 
\begin{cor52}\label{cor:52}
All chordal bipartite graphs are in $\NNN'$.
\end{cor52}

In particular, if $\Gamma$ is chordal or chordal bipartite, then $A(\Gamma)$ does not contain a closed hyperbolic surface group.
This first appeared in~\cite{kim2007} and also follows from~\cite{CSS2008}.

We will also prove that $\NNN'$ is closed under a certain graph operation, called {\em co-contraction}~\cite{kim2008}.
Using this, a lower bound for $\NNN'$ will be given. 
We will provide new examples of
 right-angled Artin groups that contain closed hyperbolic surface groups, by using co-contraction and
 results in~\cite{CSS2008}.
 Lastly,
 we will describe an equivalent formulation of Conjecture~\ref{conj:nclosedkn}. A vertex of a graph $\Gamma$ is called 
 {\em simplicial}, 
 if the link of the vertex induces a complete subgraph of $\Gamma$.

\enumir
\begin{prop64}\label{prop:64}
The following are equivalent.
\be
\item
$\NNN$ is closed under complete graph amalgamation.
\item
For any graph $\Gamma$, if  the graph obtained by removing a simplicial vertex from $\Gamma$  is in $\NNN$, 
then $\Gamma$ is also in $\NNN$.
\item
$\NNN' = \NNN$
\ee
\end{prop64}
\enumia

{\bf Note.}  All the results in this article, except for Example~\ref{exmp:nine}, originally appeared in the Ph.D. thesis of the author~\cite{kim2007}. 
After submission of his thesis, 
the author came to know that Crisp, Sageev and Sapir  proved similar results to
 Corollary~\ref{cor:chordal} and
Theorem~\ref{thm:bisimplicial}, 
where $\NNN'$ is replaced by the presumably larger class $\NNN$~\cite[Proposition 3.2 and Lemma 6.3]{CSS2008}.
A special case of Lemma~\ref{lem:promotion}
can also be deduced from~\cite[Theorem 4.2]{CSS2008}. That is the case when 
there exists a fixed complete subgraph $K\le\Gamma$ and 
a relative embedding $\phi:\pi_1(S)\rightarrow A(\Gamma)$ for some compact hyperbolic surface $S$,
such that the image of
each peripheral element of $\pi_1(S)$ is conjugate into $A(K)$.
Their work is independent from the author, and the arguments are completely different.

{\bf Acknowledgement.} 
I would like to thank my Ph.D. thesis advisor, Professor Andrew Casson for sharing his deep insights and valuable advice that guided me through this work.
I am also grateful to Professor Alan Reid for many helpful comments on earlier versions of this article.
Lastly, I am thankful for exceptionally kind and detailed feedback from an anonymous referee, particularly for hinting Remark~\ref{rem:promotion}.

\section{Label-Reading Maps} \label{sec:prelim}

In this section, we review  basic properties of label-reading maps from  surface groups into right-angled Artin groups. 
We owe most of the  definitions and the results in this section to~\cite{CW2004}. 

Recall our convention that we only consider oriented surfaces.
Let $S$ be a compact  surface, with an arbitrarily chosen base in its interior.
From the orientation of $S$, the boundary components of $S$ inherit 
 canonical orientations, so that $\sum [\partial_i S]=0$ in $H_1(S)$.
 We will often abbreviate a closed curve and a properly embedded arc on $S$
as  a {\em curve} and an {\em arc}, respectively.
We assume that a curve or an arc is given with an orientation, which is arbitrarily chosen unless specified.
Suppose $\Gamma$ is a graph.
In $A(\Gamma)$,
a {\em letter} means $v^{\pm1}$ for some $v\in V(\Gamma)$, and a {\em word} means a sequence of letters. A word represents an element in $A(\Gamma)$.  
A {\em label-reading pair on $S$ with the underlying graph $\Gamma$} is a pair $(\HH,\lambda)$ such that
  $\HH$ is a set of simple closed curves and properly embedded arcs on $S$,
  and  $\lambda\co\HH\rightarrow V(\Gamma)$, called the {\em labeling}, satisfies that
   $\alpha\cap\beta\ne\varnothing$ only if
  $[\lambda(\alpha),\lambda(\beta)]=1$ in $A(\Gamma)$. 
 Here, curves and arcs in $\HH$ with the same label are allowed to intersect. This difference from~\cite{CW2004} leaves
 most of the results and the arguments in~\cite{CW2004} still valid.
 We simply call $(\HH,\lambda)$ as a {\em label-reading pair} if its underlying graph is obvious
 from the context.
For $a\in V(\Gamma)$,  an {\em $a$--curve} and an {$a$--arc} in $\HH$
will mean a curve and an arc, respectively, labeled by $a$.
For each based loop $\gamma$ transversely intersecting $\HH$, 
one follows $\gamma$ starting from the base point; whenever $\gamma$ intersects
  $\alpha\in\HH$, one records $\lambda(\gamma)$ or $\lambda(\gamma)^{-1}$, according to whether the orientation of $\gamma$
  coincides with the transverse orientation of $\alpha$ or not. The word $w_\gamma$ thus obtained is called the {\em label-reading}
  of $\gamma$ with respect to $(\HH,\lambda)$.
  Note that the label-reading $w_\gamma$ can also be defined if $\gamma$ is
an oriented arc, instead of an oriented curve.
The map $\phi\co\pi_1(S)\rightarrow A(\Gamma)$, defined by $\phi([\gamma])=w_\gamma$,
  is
  called {\em the label-reading map} with respect to $(\HH,\lambda)$. 
  
Conversely, suppose $\phi\co\pi_1(S)\rightarrow A(\Gamma)$ is an arbitrary map. Write
\[\pi_1(S) = <x_1,x_2,\ldots,
x_g,y_1,y_2,\ldots,y_g,d_1,d_2,\ldots,d_m|
\prod_{i=1}^g[x_i,y_i] \prod_{i=1}^m d_i>.\]
Here $d_1,d_2,\ldots,d_m$ correspond to 
the boundary components $\partial_1 S,\partial_2 S, \ldots,
\partial_m S$ of $S$.
Draw a {\em dual van Kampen diagram} $\Delta$ for the following word in $A(\Gamma)$~\cite{olshanskii1989,kim2008}:
\[
w=\prod_{i=1}^g[\phi(x_i),\phi(y_i)] \prod_{i=1}^m \phi(d_i).\]
 \begin{figure}[bt!] 
 \centering
\subfigure[]{
\labellist 
\small\hair 2pt 
\pinlabel {${}_{c^{-1}}$} [b] at 18 26
\pinlabel {${}_b$} [b] at 112 45
\pinlabel {${}_{a^{-1}}$} [b] at 117 80
\pinlabel {${}_{b^{-1}}$} [b] at 65 115
\pinlabel {${}_c$} [b] at 18 93
\pinlabel {${}_a$} [b] at 65  7
\pinlabel {${}_w$} [r] at 0 36
\pinlabel {$\rightarrow$} [l] at 54 94
\pinlabel {$\leftarrow$} [r] at 43 64
\pinlabel {${}_c$} [l] at 23 65
\pinlabel {$\searrow$} [l] at 56 50
\pinlabel {${}_b$} [r] at 56 95
\pinlabel {${}_a$} [rb] at 75 38
\endlabellist 
\includegraphics{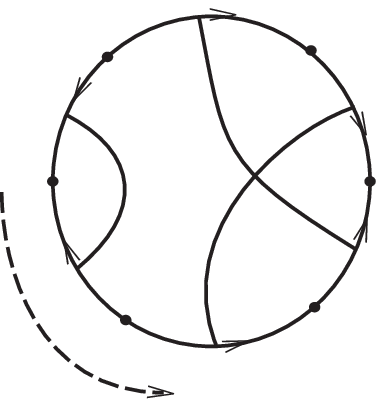}
}
\hspace{0.7in}
\subfigure[]{
 \labellist 
\small\hair 2pt 
\pinlabel {${}_{\phi(x_i)}$} [b] at 110 90
\pinlabel {${}_{\phi(y_i)}$} [b] at 110 30
\pinlabel {$\scriptstyle \Delta$}   at 60 65
\pinlabel {${}_{\phi(x_i)^{-1}}$} [b] at 59 0
\pinlabel {${}_{\phi(y_i)^{-1}}$} [b] at 5 30
\pinlabel {${\iddots}$}    at 25 116
\endlabellist 
\includegraphics{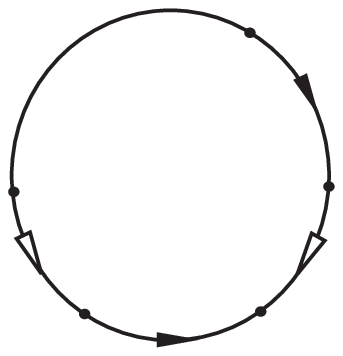}
}
\caption{(a) A dual van Kampen diagram $\Delta$ for the word $w = c^{-1}aba^{-1}b^{-1}c$ in the right-angled Artin group
$\langle a, b, c | [a,b]=1\rangle$. (b) Identifying intervals on $\partial \Delta$.
\label{fig:dualvk}}
\end{figure}
Recall, this means $\Delta$ is a disk along with a set of transversely oriented, properly embedded arcs labeled by $V(\Gamma)$, such that the label-reading of $\partial\Delta$ with respect to these arcs
 is $w$ (Figure~\ref{fig:dualvk}\;(a)).
 $\partial\Delta$ is subdivided into segments,
so that the label-reading of each segment is a letter in $V(\Gamma)^{\pm1}$.
Glue the boundary of $\Delta$ 
by identifying $\phi(x_i)$ with $\phi(x_i)^{-1}$, and also 
$\phi(y_i)$ with $\phi(y_i)^{-1}$, as in Figure~\ref{fig:dualvk}\;(b).
Then one obtains  $S$ back,
with a set $\HH$ of transversely oriented curves and  arcs on $S$ and a labeling map $\lambda\co\HH\rightarrow V(\Gamma)$.
It follows that $(\HH,\lambda)$ is a label-reading pair,
and $\phi$ is the label-reading map
with respect to $(\HH,\lambda)$ up to conjugation in $A(\Gamma)$.
Moreover, if $\phi$ is a relative embedding, then for each $i$ there exists a complete subgraph
$K\le \Gamma$ such that 
$\phi(d_i) = w_i'^{-1} w_i w_i'$ for some  $w_i\in A(K)$ and $w_i'\in A(\Gamma)$.
In this case, 
$(\HH,\lambda)$
 can be chosen so that any arc in $\HH$
intersecting with a  boundary component 
$\partial_i S$ is labeled by a letter in $V(K)$.
This is achieved by gluing the words $w_i'$ and $w_i'^{-1}$ in our construction. We summarize this as follows.

\enumia
\begin{prop}[\cite{CW2004,kim2007}]\label{prop:lrexists}
Let $S$ be a compact surface and $\Gamma$ be a graph.
Suppose $\phi\co\pi_1(S)\rightarrow A(\Gamma)$ is a map.
\be
\item
$\phi$ is a label-reading map with respect to some label-reading pair $(\HH,\lambda)$
with the underlying graph $\Gamma$.
\item
If $\phi$ is a relative embedding, then one can find $(\HH,\lambda)$ in (1), satisfying the following:
for each boundary component $\partial_i S$,
there exists a complete subgraph $K\le \Gamma$ such that 
any arc in $\HH$ intersecting with $\partial_i S$ is labeled by a vertex in $V(K)$.
\ee
\end{prop}

From now on, whenever we are given with a relative embedding $\phi:\pi_1(S)\rightarrow A(\Gamma)$
with respect to a label-reading pair $(\HH,\lambda)$, we will implicitly assume that $(\HH,\lambda)$
satisfies the property described in Proposition~\ref{prop:lrexists} (2).

\enumia
\begin{notation}\label{notation:hom}
Let $S$ be a compact surface.
\be
\item  If $\alpha$ and $\beta$ are closed curves, $\alpha\sim\beta$ means 
$\alpha$ and $\beta$ are freely homotopic.
\item
If $\alpha$ and $\beta$ are properly embedded arcs, $\alpha\sim\beta$ means
$\alpha$ and $\beta$ are homotopic, by a homotopy leaving the endpoints of $\alpha$ and $\beta$
on the boundary of $S$ (but not requiring the endpoints to be fixed).
\item
Let $A\subseteq S$.
We write $\alpha\leadsto A$ and say $\alpha$ is {\em homotopic into $A$},
if $\alpha\sim\beta$ for some  $\beta\subseteq A$. In particular, if $A=\partial S$, we say that $\alpha$ is homotopic into the boundary
of $S$.
\item
$i(\alpha,\beta) = \min\{\alpha'\cap\beta'\;|\; \alpha\sim\alpha'\mbox{ and }\beta\sim\beta'\}$.
\ee
\end{notation}

We say two maps $\phi,\psi\co\pi_1(S)\rightarrow A(\Gamma)$ are {\em equivalent} if $\phi=i\circ\psi$ for some inner-automorphism
$i\co A(\Gamma)\rightarrow A(\Gamma)$. Note that for a fixed label-reading pair, a base change does not alter the equivalence class of the corresponding label-reading map $\pi_1(S)\rightarrow A(\Gamma)$. We also say two label-reading pairs are {\em equivalent} if they induce equivalent label-reading maps.
 There are certain simplifications on $(\HH,\lambda)$ that do not change the equivalence class (see~\cite{CW2004} for details and proofs). 

\begin{lem}[\cite{CW2004}]\label{lem:equiv}
Let $(\HH,\lambda)$ and $(\HH',\lambda')$ be label-reading pairs on $S$ with the underlying graph $\Gamma$. The label-reading maps induced by
$(\HH,\lambda)$ and $(\HH',\lambda')$ are equivalent, if
any of the following is satisfied.
\be
\item
$\HH'$ is obtained by removing null-homotopic curves in $\HH$.
\item
$\HH'$ is obtained by removing $\alpha\in\HH$, for some $\alpha\leadsto\partial S$.
\item
Suppose  $\alpha,\beta\in\HH$ intersect at $p$ and have the same label $a$.
Alter  $\alpha$ and $\beta$ on a neighborhood $D$ of $p$; we can get $\alpha'$ and $\beta'$
which are labeled by $a$ and do not intersect in $D$.
Transverse orientations of $\alpha'$ and $\beta'$ are determined by those of $\alpha$ and 
$\beta$.
$(\HH',\lambda')$ is the label-reading pair thus obtained  (Figure~\ref{fig:reduction}\;(a)).
\item
Suppose  $\alpha,\beta\in\HH$ bound a bigon.
Alter  $\alpha$ and $\beta$ on a neighborhood of the bigon;
we can get $\alpha'$ and $\beta'$
which do not intersect
in that neighborhood. The labels of $\alpha'$ and $\beta'$ are equal to those of $\alpha$ and $\beta$, respectively.
$(\HH',\lambda')$ is the label-reading pair thus obtained  (Figure~\ref{fig:reduction}\;(b)).\qed
\ee
\end{lem}

\begin{figure}[htb!]
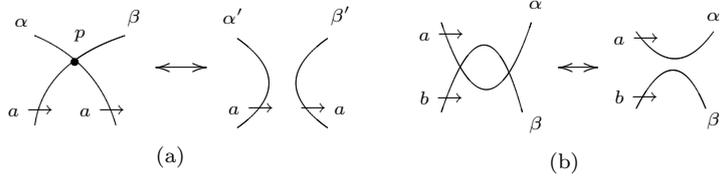

\centering
\subfigure[]{
$$
\xy 
(0,0)*++
{\xy 
0;/r.14pc/:
(10,10)*{}; (-10,-10)*{} **\crv{(-8,4)};
(-10,10)*{}; (8,-10)*{} **\crv{(4,4)};
(-11,-7)*{{\scriptstyle a}\rightarrow};
(5,-7)*{  {\scriptstyle a}\rightarrow};
(0,10)*{{\scriptstyle p}};
(-1,4)*{{}_\bullet};
(12,14)*{ {\scriptstyle  \beta} };
(-13,13)*{{\scriptstyle  \alpha}};
\endxy  }="x";
(28,0)*++
{\xy
0;/r.14pc/:
(10,10)*{}; (10,-10)*{} **\crv{(-4,0)};
(-10,-10)*{}; (-10,10)*{} **\crv{(4,0)};
(-7,-6)*{  {\scriptstyle a}\rightarrow};
(9,-6)*{  \rightarrow  {\scriptstyle a}};
(13,15)*{   {\scriptstyle \beta'} };
(-11,15)*{ {\scriptstyle \alpha'}};
\endxy  }="y";
{\ar@{<->} "x";"y"}; 
\endxy
$$
}
\subfigure[]{
$$
\xy
(0,0)*++
{\xy
0;/r.14pc/:
(10,10)*{}; (-10,10)*{} **\crv{(0,-20)};
(-10,-10)*{}; (8,-10)*{} **\crv{(0,20)};
(-10,-7)*{  {\scriptstyle b}\rightarrow};
(-10,7)*{  {\scriptstyle a}\rightarrow};
(11,14)*{ {\scriptstyle  \alpha}};
(11,-13)*{  {\scriptstyle \beta}};
\endxy  }="z"; 
(25,0)*++
{\xy 
0;/r.12pc/:
(10,10)*{}; (-10,10)*{} **\crv{(0,-4)}; 
(-10,-10)*{}; (8,-10)*{} **\crv{(0,10)}; 
(-10,-7)*{  {\scriptstyle b}\rightarrow};
(-10,8)*{  {\scriptstyle a}\rightarrow};
(11,15)*{  {\scriptstyle \alpha'}};
(11,-13)*{{\scriptstyle   \beta'}};
\endxy  }="w"; 
{\ar@{<->} "z";"w"};
\endxy 
$$
}
\caption[Homotopies on a label-reading pair]
{Homotopies that do not
change the equivalence class
of a label-reading pair.
Note that 
(b) is allowed only when $a=\lambda(\alpha)$ and $b=\lambda(\beta)$ are equal or adjacent in $\Gamma$.
\label{fig:reduction}}
\end{figure}
\begin{rem}\label{rem:reduction}
Let  $(\HH,\lambda)$ be  a label-reading pair on a compact surface $S$, inducing a label-reading map $\phi\co\pi_1(S)\rightarrow A(\Gamma)$.
\be
\item
By Lemma~\ref{lem:equiv}, we will always assume that curves and arcs in $\HH$ are neither null-homotopic
nor homotopic into the boundary. Moreover, curves and arcs in $\HH$ 
will be assumed to be minimally intersecting~\cite{CB1988}.
Curves and arcs of the same label are assumed to be disjoint, unless stated otherwise.
\item
Let $\gamma$ be a curve or an arc on $S$, such that 
the endpoints (meaning the base point if $\gamma$ is a loop) are not on $\HH$.
If $\gamma$ is not transverse to $\HH$, we set $w_\gamma=w_{\gamma'}$ for some $\gamma'\sim\gamma$
such that $\gamma'$ has the same endpoints as $\gamma$, and 
$\gamma'$ is transverse to $\HH$. This definition of $w_\gamma \in A(\Gamma)$ does not depend on the choice of $\gamma'$.
\ee
\end{rem}

\section[$\NNN'$ is closed under $K_n$-amalgamation]{$\NNN'$ is closed under complete graph amalgamation}

Recall that $\NNN'$ denotes the class of graphs, the right-angled Artin groups on which do not allow
relative embeddings of compact hyperbolic surface groups (Definition~\ref{defn:relative}). 
Let $\Gamma$ be a graph and $S$ be a compact hyperbolic surface. 
Recall that $x\in\pi_1(S)$ is called {\em peripheral} if $x=[\alpha]$ for some $\alpha$ homotopic into $\partial S$.
We note that 
an embedding $\phi\co \pi_1(S)\rightarrow A(\Gamma)$ 
is a relative embedding if, for each peripheral $x\in \pi_1(\Gamma)$, 
there exists a complete subgraph $K\le\Gamma$ such that $\phi(x)$ is conjugate into $A(K)$.
In this section, we examine basic combinatorial properties
of $\NNN'$, and prove that $\NNN'$ is closed under complete graph amalgamation. 
Roughly speaking, a  key idea for the proof is that
 commutativity is  scarce in surface groups. The following is immediate from the fact that
 any 2-generator subgroup of a compact hyperbolic surface group
 is free.
 
\begin{lem}\label{lem:primitive}
Let $S$ be a compact hyperbolic surface, and  $x$ and $y$ be commuting elements of $\pi_1(S)$.
Then there exists $c\in\pi_1(S)$ such that
$x,y\in\langle c\rangle$. If $x$ and $y$ are further assumed to be represented by essential simple closed curves,
then either $x=y$ or $x=y^{-1}$.\qed
\end{lem}

Let $\Gamma_1$ and $\Gamma_2$ be graphs.
The disjoint union of $\Gamma_1$ and $\Gamma_2$ is
denoted by $\Gamma_1\sqcup\Gamma_2$. 
We define $\join(\Gamma_1,\Gamma_2)$ to be the graph obtained by
taking the disjoint union of $\Gamma_1$ and $\Gamma_2$
and adding the edges in 
$\{\{v_1,v_2\} \; | \; v_1\in V(\Gamma_1),v_2\in V(\Gamma_2)\}$.
This means,
\[\join(\Gamma_1,\Gamma_2) = \overline{\overline{\Gamma_1}\sqcup\overline{\Gamma_2}}.\]

\begin{prop} \label{prop:join}
If $\Gamma_1,\Gamma_2\in\NNN'$, then
$\join(\Gamma_1,\Gamma_2)\in\NNN'$.
\end{prop}

\begin{proof}
Suppose $\Gamma = \join(\Gamma_1,\Gamma_2)\not\in\NNN'$.
One can find a relative embedding
$\phi\co \pi_1(S)\rightarrow A(\Gamma)\cong A(\Gamma_1)\times A(\Gamma_2)$ for some hyperbolic surface $S$.
Let  $(\HH,\lambda)$ be a label-reading pair inducing $\phi$, and
 $p_i:A(\Gamma)\rightarrow A(\Gamma_i)$ be the projection map.

We claim that $p_1\circ\phi$ or $p_2\circ\phi$ is injective.
Suppose not, and choose 
$1\ne a_1\in\ker (p_1\circ\phi)$ and $1\ne a_2\in\ker(p_2\circ
\phi)$. Write $\phi(a_1)=(1,b_2)$ and $\phi(a_2)=(b_1,1)$ for some non-trivial 
$b_i\in A(\Gamma_i)$, $i=1,2$.
$\phi[a_1,a_2]=[\phi(a_1),\phi(a_2)] = [(1,b_2),(b_1,1)]=1$.
Since $S$ is hyperbolic and $\phi$ is an embedding, $a_1,a_2\in\langle c\rangle$ for some $c\in\pi_1(S)$ 
(Lemma~\ref{lem:primitive}).
Hence $\langle  \phi(a_1),\phi(a_2)\rangle \subseteq \langle  \phi(c)\rangle\cong\Z$,
which contradicts to
$\Z\times\Z \cong \langle  (1,b_2),(b_1,1)\rangle = \langle  \phi(a_1),\phi(a_2)\rangle $

Without loss of generality, we may assume that
$p_1\circ\phi$ is injective.
The label-reading map $\pi_1(S)\rightarrow A(\Gamma_1)$
obtained by removing curves and arcs in $\HH$ labeled by $V(\Gamma_2)$
is injective. So $\Gamma_1\not\in\NNN'$. \end{proof}

 Since $K_1\in\NNN'$, it follows from Proposition~\ref{prop:join} that $K_n\in\NNN'$ for any $n$. 
This can also be seen by noting that $A(K_n)\cong \Z^n$ does not contain a closed hyperbolic surface group
 or a non-abelian free group. 
 
 For the rest of this section, we will prove that $\NNN'$ is closed under complete graph amalgamation.
For a graph $\Gamma$, the set of all vertices adjacent to $a\in V(\Gamma)$ will be denoted by $\link(a)$.

\begin{defn}[\cite{golumbic2004}] \label{def:simplicialv}
A vertex $a$ of a graph $\Gamma$ is called {\em simplicial}
if $\link(a)$ induces a complete subgraph of $\Gamma$.
\end{defn}

A set of pairwise non-adjacent vertices in a graph is said to be {\em independent}.

\begin{lem}\label{lem:simplicialv}
Let $\Gamma$ and $\Gamma'$  be graphs such that
$\Gamma'$
is obtained by removing a set of independent  simplicial vertices in $\Gamma$.
 If $\Gamma'\in\NNN'$, then $\Gamma\in\NNN'$.
 \end{lem}

\begin{proof}
Let
$\Gamma'$ be the induced subgraph of $\Gamma$ on
$V(\Gamma)\setminus\{a_1,\ldots,a_r\}$,
where
 $a_1,\ldots,a_r$ are independent simplicial vertices of $\Gamma$.
Suppose $\Gamma\not\in\NNN'$. 

First, consider the case when $r=1$. Let $a=a_1$.
One can find a compact hyperbolic surface $S$ and 
a relative embedding    $\phi\co \pi_1(S)\rightarrow A(\Gamma)$ induced by
a label-reading pair $(\HH,\lambda)$. Put $\HH_a = \lambda^{-1}(a)$.

{\em Case 1. $\HH_a$ consists of simple closed curves only.}

Choose a connected component $S'$ of $S\setminus(\cup\HH_a)$, so that 
$S'$ is hyperbolic. The curves and arcs in the set $(\cup\HH)\cap S'$ naturally
inherit transverse orientations and labels from those of $(\HH,\lambda)$,
and so determine a label-reading pair $(\HH',\lambda')$ inducing
$\phi':\pi_1(S')\rightarrow A(\Gamma')$.
$\phi'$ is injective, since $\phi'$ is a restriction of $\phi$ up to equivalence.

A simple closed curve in $\HH_a$ intersects with a curve in $\HH$ labeled by
a vertex in $\link(a)$. 
Each  boundary component $\partial_i S'$
of $S'$ either is 
a boundary component of $S$, or comes from a curve in $\HH_a$.
In the latter case, any curve in $\HH'$ intersecting with 
$\partial_i S'$ must be labeled
by  a  vertex in   $\link(a)$. Since $\link(a)$ induces a complete graph in $\Gamma'$, 
$\phi'$ is a relative embedding. This implies
 $\Gamma'\not\in\NNN'$.

{\em Case 2. $\HH_a$ contains a  properly embedded arc $\alpha$.}

Suppose $\alpha$ joins the boundary components $\partial_1 S$ and $\partial_2 S$.
Here, $\alpha$ is an essential arc, but it is possible that $\partial_1 S = \partial_2 S$.
Since $\phi$ is a relative embedding, any curve or arc in $\HH$ intersecting with $\partial_1 S$  or 
$\partial_2 S$ is labeled by a vertex in $\link(a)\cup\{a\}$.
 From the definition of a label-reading pair, the label-reading of  $\alpha$ is in $\langle  \link(a)\rangle$. 
Choose $\delta_1\sim \partial_1 S$
and $\delta_2 \sim \partial_2 S$, 
such that $\delta_1$ and $\delta_2$ have the same basepoint, and transversely intersect $\HH$.
Moreover, we assume that  $\delta_1$ and $\delta_2$ are sufficiently close to $\partial_1 S$ and $\alpha\cdot\partial_2 S\cdot\alpha^{-1}$
respectively,  so that $w_{\delta_1}$ and $w_{\delta_2}$ are in $\langle  \link(a)\cup\{a\}\rangle$
(Figure~\ref{fig:lem:simplicialv}).
Since $a$ is simplicial, $\langle  \link(a)\cup\{a\}\rangle$ is free abelian,
and 
$\phi([[\delta_1],[\delta_2]]) = [w_{\delta_1},w_{\delta_2}] = 1$.
The injectivity of $\phi$ implies
$[[\delta_1],[\delta_2]]=1$, which  is impossible unless
$\delta_1\sim \delta_2^{\pm1}$ and $S$ is an annulus (Lemma~\ref{lem:primitive}).

In the case when $r>1$,
note that $a_r$ is a simplicial vertex of the induced subgraph on 
 $V(\Gamma)\setminus\{a_1,a_2,\ldots,a_{r-1}\}$.
An inductive argument shows that $\Gamma'\not\in\NNN'$. \end{proof}

 \begin{figure}[htb!] 
 \labellist 
\small\hair 2pt 
\pinlabel {$\scriptstyle\partial_1 S$} [b] at 29 30
\pinlabel {$\scriptstyle\partial_2 S$} [b] at 125 30
\pinlabel {$\scriptstyle\alpha$} [b] at 92 52
\pinlabel {$\uparrow$} [t] at 70 55
\pinlabel {$\scriptstyle a$} [b] at 70 55
\pinlabel {$\scriptstyle\delta_2$} [l] at 157 35
\pinlabel {$\scriptstyle\delta_1$} [r] at 0 36
\pinlabel {$\bullet$} [r] at 57 21
\endlabellist 
\centering 
\includegraphics{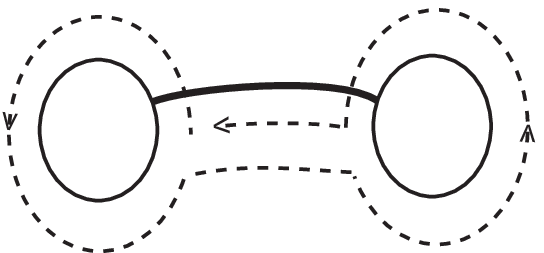}
\caption{Proof of Lemma~\ref{lem:simplicialv}
\label{fig:lem:simplicialv}}
\end{figure}

Using the next two lemmas, 
we will prove a general fact (Lemma~\ref{lem:evinj}) on the fundamental group of a hyperbolic surface with boundary.
For a set $X$, the {\em period} of a finite sequence $f:\{1,2,\ldots,M\}\rightarrow X$ is the {\em smallest} positive number $p$ such that
$f(i) = f(i+p)$ whenever $ i$ and $ i + p$ are in $\{1,2,\ldots, M\}$. The following combinatorial lemma asserts that if two finite sequences 
coincide at consecutive terms the number of which is larger than the sum of the periods, 
then one sequence is a translation of the other.

\begin{lem}\label{lem:period}
Let $X$ be a set, and $M_1,M_2>0$. For $i=1,2$, let $A_i= \{1,2,\ldots,M_i\}$, and $f_i:A_i\rightarrow X$ be a finite sequence with the period $p_i$.
Suppose there exist integers $u$ and $v$ such that
 for each $i = v+1, v+2, \ldots, v+ p_1 + p_2$,
we have $i\in A_1$, $u+i \in A_2$ and $f_1( i ) = f_2( u + i )$.
Then $p_1 = p_2$, and $f_1(i)  = f_2(u+i)$ whenever $i\in A_1$, $u+i \in A_2$.
\end{lem}

\begin{proof}
We may assume $p_1\le p_2$. 
Suppose $i,i+p_1\in A_2$.
There exists some $q$ such that 
$u+v+1\le i + p_2 q\le u+v+p_2$. Then
$f_2(i+p_1) = f_2(i + p_1 + p_2 q)
= f_1(i +p_1 +p_2 q - u )
= f_1(i+ p_2 q - u) 
= f_2(i+p_2 q)
=f_2(i)
$.
Note that we used the conditions that $v+1+p_1\le i+ p_1+p_2q-u \le v+p_1+p_2$
and $v+1\le i + p_2 q - u\le v+p_2$. This shows that $p_1$ is also the period of $f_2$, and so, $p_1=p_2$. Now suppose
$i\in A_1$ and $u+i\in A_2$.
For some $q'$, $v+1\le i+p_2 q' \le v+p_2$. 
Hence, $f_1(i) =  f_1(i+p_2q') = f_2(u+i+p_2q') = f_2(u+i)$.
\end{proof}

Recall that two elements in a free group are said to be {\em independent}
if they do not have non-trivial conjugate powers.
In the following lemma,
the special case when $u_1=u_2=\ldots=u_m$ is first proved in~\cite{baumslag1962}. This case was further generalized to any word-hyperbolic group in~\cite{GW2007}. The following proof uses a similar idea to~\cite{GW2007}.

\begin{lem}\label{lem:evinj0}
Let $F$ be a free group. 
Suppose $u_1,\ldots, u_m\in F\setminus\{1\}$, satisfying that any pair $u_i$ and $u_j$ are either equal or independent.
Set $u_0 = u_m$.
Choose $b_1,\ldots, b_m\in F$ such that $u_{i-1} = u_i$ only if $[b_i,u_i]\ne1$.
Then there exists $N>0$ such that for any $|n_1|,\ldots,|n_m|>N$,
$b_1 u_{1}^{n_1} b_2 u_{2}^{n_2} \cdots b_{m} u_{m}^{n_m}$
is non-trivial in $F$.
\end{lem}

\begin{proof}
We may assume that each $u_i$ is cyclically reduced and not a proper power.
For $g\in F$, $|g|$ denotes the word-length of $g$.
Let $N$ be a sufficiently large integer which will be determined later in the proof, 
and $|n_1|,\ldots,|n_m|>N$. Suppose $ w = b_1 u_{1}^{n_1} b_2 u_{2}^{n_2} \cdots b_m u_m^{n_m}$ is trivial in $F$.
Consider a dual van Kampen diagram $\Delta$ of $w$, which is a disk along with disjoint properly embedded arcs.  The boundary
$\partial\Delta$ is divided into segments, each of which intersects with only one properly embedded arc.
For each $i$, the interval $u_i^{n_i}$ on $\partial\Delta$ intersects with $|n_i | |u_i|$ arcs. Since $u_i$ is cyclically reduced,
no arc intersects $u_i^{n_i}$ twice. Hence there exists $j$ such that there are at least
\[
M =  \frac1n ( |n_i| |u_i| - \sum_k |b_k|)
\]
arcs joining $u_i^{n_i}$ and $u_j^{n_j}$. If two arcs $\alpha$ and $\beta$
join $u_i^{n_i}$ and $u_j^{n_j}$, then so does any arc between $\alpha$ and $\beta$. 
This means $u_i^{n_i}$ and $u_j^{-n_j}$ have a common subword of length at least $M$.
The word $u_i^{n_i}$ is a finite sequence of the period $| u_i |$. Since $|u_i|\ne0$, one can choose a sufficiently large $N$
such that $M > \sum_k |u_k|$. Lemma~\ref{lem:period} implies that $u_i$ is a cyclic conjugation of $u_j^{\pm1}$. By the independence of
$u_i$ and $u_j$, $u_i = u_j$. 
Note that such $j$ exists for any $i$. So, if one chooses such a pair $(i,j)$ which is innermost,
then $j=i+1$ or $j=i-1$. Assume $j=i-1$.
In $\Delta$, some arcs join an interval of the form $u_i^k$ in $u_i^{n_i}$ to an interval in $u_{i-1}^{n_{i-1}}$.
By cutting $\Delta$ along these arcs, one obtains another dual van Kampen diagram for some word of the form $u_{i-1}^{p} b_i u_i^{q} = u_i^p b_i u_i^q$.
Here, $u_{i-1}^p$ and $u_i^q$ are subwords of
 of $u_{i-1}^{n_{i-1}}$
 and $u_i^{n_i}$, respectively.
We have $ u_i^p b_i u_i^q=1$, which is a contradiction to $[b_i,u_i]\ne1$.
\end{proof}

For a compact surface $S$, 
$D(S)$ denotes the double of $S$ along $\partial S$.
The following lemma is well-known when the surface $S$ has only one boundary component~\cite{wilton2006,baumslag1962}.

\begin{lem}\label{lem:evinj}
Let $S$ be a surface with the boundary components
$\partial_1 S,\ldots,\partial_m S$.
 $q:D(S)\rightarrow S$ denotes the natural
quotient map. 
Let $T_i:D(S)\rightarrow D(S)$
be the Dehn twist along $\partial_i S\subseteq D(S)$.
Then for any $x\in\pi_1(D(S))$
there exists $N>0$ 
 such that whenever $|n_1|,\ldots,|n_m|>N$,
$(q\circ T_1^{n_1}\circ\cdots\circ T_m^{n_m})_*(x)\ne\{1\}$.
\end{lem}

\begin{proof}
Put $\pi_1(S) = \langle x_1,x_2,\ldots,x_g,y_1,y_2,\ldots,y_g,d_1,d_2,\ldots,d_m\; | \;
\prod[x_i,y_i]\prod d_i = 1 \rangle$, 
where $d_i$ is represented by a loop freely homotopic to $\partial_i S$.
Let $h:S'\rightarrow S$ be a homeomorphism, such that $D(S) = S\cup S'$ glued along
$\partial_i S = \partial_i S'$ for each $i$.
Let $v$ be the base point of $S$, and $v'$ be its image in $S'$. 
One can find arcs $\delta_1,\delta_2,\ldots,\delta_m$  joining $v$ to $v'$ in $D(S)$
such that (Figure~\ref{fig:double}),
\enumir
\be
\item
$\delta_i$ and $\partial_j S$ intersect if and only if $i=j$,
\item
$
[q\circ T_1^{n_1}\circ\cdots\circ T_m^{n_m}(\delta_i)]
=
[q\circ T_i^{n_i}(\delta_i)]
= d_i^{n_i}$.
\ee
\enumia

\begin{figure}[htb!] 
\labellist 
\small\hair 2pt 
\pinlabel {${}_v$}  [rt] at 44 27
\pinlabel {${}_{v'}$} [lt] at 133 31
\pinlabel {$\scriptstyle S$} [rb] at 13 60
\pinlabel {$\scriptstyle S'$} [lb] at 156 60
\pinlabel {$\scriptstyle {\partial_2} S$} [t] at 84 -1
\pinlabel {${}_{{\partial_1} S}$} [b] at 84 75
\pinlabel {$\scriptstyle {\delta_1}$} [lb] at 111 49
\pinlabel {$\scriptstyle {\delta_2}$} [lt] at 116 18
\endlabellist 
\centering
\includegraphics{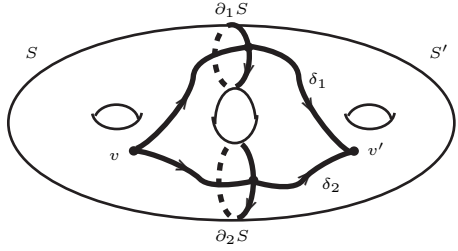}
\caption{The double of a surface. \label{fig:double}}
\end{figure}

Choose any $x\in \pi_1(D(S))$. For some $l\ge0$ and $1\le i_1,i_2,\ldots,i_{2l} \le m$, 
$x$ can be represented as a concatenation of arcs : 
$x = \beta_1\cdot \delta_{i_1}\cdot \beta_2\cdot \delta_{i_2}^{-1}\cdots \beta_{2l}\cdot \delta_{i_{2l}}^{-1}$.
Here, $\beta_1,\beta_3,\beta_5,\ldots$ are loops in $S$ based at $v$,
and  $\beta_2,\beta_4,\beta_6,\ldots$ are loops in $S'$ based at $v'$.
By choosing the minimal $l$, one may assume that if $i_{k-1} = i_k$ then $\beta_k$ is not homotopic into $\partial_{i_k} S = \partial_{i_k} S'$. 
One can write
\[ 
(q\circ T_1^{n_1}\circ\cdots\circ T_m^{n_m})_*(x)
=
 [\beta_1]d_{i_1}^{n_{i_1}}
 [h(\beta_2)] d_{i_2}^{-n_{i_2}}
  \cdots  [h(\beta_{2l})] d_{i_{2l}}^{-n_{i_{2l}}}.\]
By applying Lemma~\ref{lem:evinj0} to the free group $\pi_1(S)$,
one sees that 
$(q\circ T_1^{n_1}\circ\cdots\circ T_m^{n_m})_*(x)\ne1$
for sufficiently large $n_i$.
 \end{proof}

From now on, we denote the set of {\em maximal} complete subgraphs of $\Gamma$ by $\KK(\Gamma)$.
We define a graph operation, called {\em simplicial extension}.

\enumir
\begin{defn} \label{defn:simplicialextension}
Let $\Gamma$ be a  graph. Define  the {\em simplicial extension of $\Gamma$},
denoted by $\Gamma^*$, to be the graph having the following vertex and edge sets.
\be
\item
$V(\Gamma^*)=V(\Gamma)\sqcup\{v_{K,u}|K\in\KK(\Gamma),\;u\in V(K)\}$
\item
$E(\Gamma^*)
= E(\Gamma) \sqcup
\{\{v_{K,u},u'\}|K\in\KK(\Gamma),\;u,u'\in V(K)\}$
\ee
\end{defn}
\enumia

$\Gamma^*$ is obtained from $\Gamma$ by adding a  simplicial vertex (denoted by
$v_{K,u}$) for each pair of a maximal complete subgraph $K$ and a vertex $u$ of $K$;
see Figure~\ref{fig:dinfty}.
We first make a graph theoretical observation regarding simplicial extensions.

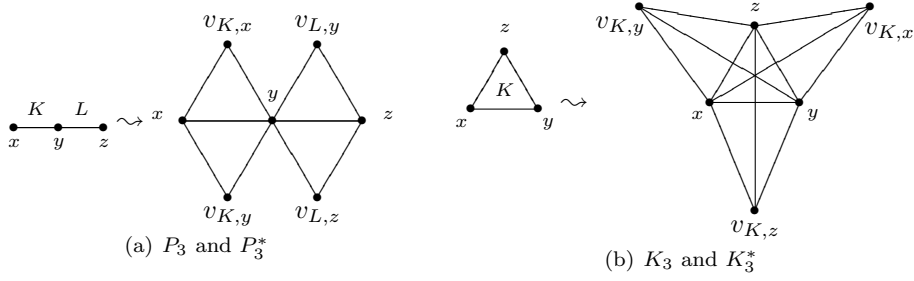
\begin{figure}[htb!] 
\subfigure[$P_3$ and $P_3^*$]{
$$
\xymatrix{
\xy
0;/r.10pc/:
(-14,0)*{}="x"; 
(-14,-5)*{\scriptstyle x};
(0,0)*{}="y"; 
(0,-5)*{\scriptstyle y};
(14,0)*{}="z"; 
(14,-5)*{\scriptstyle z};
(-7,6)*{\scriptstyle K};
(7,6)*{\scriptstyle L};
"x"*{{}_\bullet};
"y"*{{}_\bullet};
"z"*{{}_\bullet};
"x";"y"**\dir{-};
"z";"y"**\dir{-};
"z";"x"**\dir{-};
\endxy
\leadsto
\xy
0;/r.20 pc/:
(-14,0)*{}="x"; 
(-18,1)*{\scriptstyle x};
(0,0)*{}="y"; 
(0,4)*{\scriptstyle y};
(14,0)*{}="z"; 
(18,1)*{\scriptstyle z};
(-7,12)*{}="vx";
(-7,15)*{v_{K,x}};
(7,12)*{}="vy";
(7,15)*{v_{L,y}};
(-7,-12)*{}="vyy";
(-7,-15)*{v_{K,y}};
(7,-12)*{}="vz";
(7,-15)*{v_{L,z}};
"x"*{{}_\bullet};
"y"*{{}_\bullet};
"z"*{{}_\bullet};
"vx"*{{}_\bullet};
"vy"*{{}_\bullet};
"vyy"*{{}_\bullet};
"vz"*{{}_\bullet};
"x";"y"**\dir{-};
"z";"y"**\dir{-};
"z";"x"**\dir{-};
"x";"vx"**\dir{-};
"x";"vyy"**\dir{-};
"z";"vy"**\dir{-};
"z";"vz"**\dir{-};
"y";"vy"**\dir{-};
"y";"vz"**\dir{-};
"y";"vx"**\dir{-};
"y";"vyy"**\dir{-};
\endxy
}
$$
}
\hfill
\subfigure[$K_3$ and $K_3^*$]{
$$
\xymatrix{
\xy
0;/r.15pc/:
(-14,0)*{}="x"; 
(-16,-3)*{\scriptstyle x};
(0,0)*{}="y"; 
(2,-3)*{\scriptstyle y};
(-7,12)*{}="z"; 
(-7,17)*{\scriptstyle z};
(-7,4)*{\scriptstyle K};
"x"*{{}_\bullet};
"y"*{{}_\bullet};
"z"*{{}_\bullet};
"x";"y"**\dir{-};
"z";"y"**\dir{-};
"z";"x"**\dir{-};
\endxy
\leadsto
\xy
0;/r.20 pc/:
(-14,0)*{}="x"; 
(-16,-2)*{\scriptstyle x};
(0,0)*{}="y"; 
(2,-2)*{\scriptstyle y};
(-7,12)*{}="z";
(-7,15)*{\scriptstyle z};
(-7,-17)*{}="vz";
(-7,-20)*{v_{K,z}};
(11,15)*{}="vx";
(14,11)*{v_{K,x}};
(-25,15)*{}="vy";
(-28,12)*{v_{K,y}};
"x"*{{}_\bullet};
"y"*{{}_\bullet};
"z"*{{}_\bullet};
"vx"*{{}_\bullet};
"vy"*{{}_\bullet};
"vz"*{{}_\bullet};
"x";"y"**\dir{-};
"z";"y"**\dir{-};
"z";"x"**\dir{-};
"x";"vx"**\dir{-};
"y";"vx"**\dir{-};
"z";"vx"**\dir{-};
"x";"vy"**\dir{-};
"y";"vy"**\dir{-};
"z";"vy"**\dir{-};
"x";"vz"**\dir{-};
"y";"vz"**\dir{-};
"z";"vz"**\dir{-};
\endxy
}
$$
}
\caption{Examples of simplicial extensions. \label{fig:dinfty}}
\end{figure}

\begin{lem}[decomposing $\Gamma^*$] \label{lem:simplicial_extension}
Suppose $\Gamma$ is a  complete graph amalgamation of $\Gamma_1$ and $\Gamma_2$.
Then $\Gamma^*$ is a complete graph amalgamation of $\Gamma_1'$ and $\Gamma_2'$
such that for each $i$,
\enumir
\be
\item
$\Gamma_i\le\Gamma_i'\le\Gamma_i^*$,
\item
$V(\Gamma_i')\setminus V(\Gamma_i)$ is a set of independent simplicial vertices of $\Gamma_i'$.
\ee
\enumia
\end{lem}

\begin{proof}
Let $K=\Gamma_1\cap\Gamma_2$. We may assume that $K\ne\varnothing$.
Note that $\KK(\Gamma_1)\cap\KK(\Gamma_2)\subseteq\{K\}$ and 
$\KK(\Gamma) \subseteq \KK(\Gamma_1)\cup \KK(\Gamma_2) \subseteq \KK(\Gamma) \cup\{K\}$

{\em Case 1. $K\in \KK(\Gamma_1)\cup\KK(\Gamma_2)$.}

We may assume $K\in\KK(\Gamma_1)$. 
Then, $K$ is maximal in $\Gamma_2$ if and only if $K$ is maximal in $\Gamma$. 
Combining this with $\KK(\Gamma_2)\subseteq
\KK(\Gamma_1)\cup \KK(\Gamma_2) \subseteq \KK(\Gamma) \cup\{K\}$, 
one has
$\KK(\Gamma_2)\subseteq \KK(\Gamma)$.
Define
 $\Gamma_1'$ to be the graph obtained from $\Gamma_1^*$ by
removing the simplicial vertices $\{ v_{K,u}\; | \; u\in V(K)\}$. Put 
$\Gamma_2'=\Gamma_2^*$. 
$\Gamma_1'\cap\Gamma_2' = \Gamma_1\cap\Gamma_2=K$.
Moreover, 
$V(\Gamma^*)\subseteq V(\Gamma_1'\cup\Gamma_2')\cup\{v_{K,u}\;|\;u\in V(K)\}$.
Since
 $\KK(\Gamma_1)\setminus\{K\}\subseteq \KK(\Gamma)$ and $\KK(\Gamma_2)\subseteq\KK(\Gamma)$,
$\Gamma_1'\cup\Gamma_2' \subseteq \Gamma^*$. 
If $K\not\in\KK(\Gamma)$, then $v_{K,u}\not\in V(\Gamma^*)$ for each $u\in V(K)$, and so,
$V(\Gamma^*)\subseteq V(\Gamma_1'\cup\Gamma_2')$.
If $K\in\KK(\Gamma)$, then $K\in\KK(\Gamma_2)$ and  $v_{K,u}\in V(\Gamma_2')$ for each $u\in V(K)$;
this implies $V(\Gamma^*)\subseteq V(\Gamma_1'\cup\Gamma_2')\cup\{v_{K,u}\;|\;u\in V(K)\}
=V(\Gamma_1'\cup\Gamma_2')$.
It follows that $ \Gamma^* = \Gamma_1'\cup\Gamma_2' $.

{\em Case 2. $K\not\in\KK(\Gamma_1)\cup\KK(\Gamma_2)$}

In this case, $\KK(\Gamma) = \KK(\Gamma_1)\sqcup \KK(\Gamma_2)$.
Hence, $\Gamma^* = \Gamma_1^*\cup\Gamma_2^*$ 
and $\Gamma_1^*\cap\Gamma_2^* = \Gamma_1\cap\Gamma_2 = K$. 
Set $\Gamma_1' = \Gamma_1^*$
and
$\Gamma_2' = \Gamma_2^*$.
 \end{proof}

Lemma~\ref{lem:promotion} is a key step for the proof of Theorem~\ref{thm:completeamalgam}. 
The lemma states that
a relative embedding of a compact hyperbolic surface group into $A(\Gamma)$ can be ``promoted'' to 
an embedding of a closed hyperbolic surface group into $A(\Gamma^*)$.

\begin{lem} \label{lem:promotion}
Let $\Gamma$ be a graph. Then $\Gamma\in\NNN'$ if and only if $\Gamma^\ast\in \NNN$.
\end{lem}

\begin{proof}
$(\Rightarrow)$
Suppose $\Gamma\in\NNN'$.
$\Gamma$ is obtained from $\Gamma^*$ by removing a set of independent simplicial vertices.
By Lemma~\ref{lem:simplicialv},
$\Gamma^*\in\NNN'\subseteq\NNN$.

$(\Leftarrow)$
Suppose $\Gamma\not\in\NNN'$. Fix a compact hyperbolic surface $S$ and a relative embedding
$\phi:\pi_1(S)\rightarrow A(\Gamma)$, which is a label-reading map with respect to  
 $(\HH,\lambda)$. Denote the boundary components of $S$ by $\partial_1 S, \partial_2 S,\ldots,\partial_m S$.
Recall $\partial S$ is oriented so that $\sum[\partial_i S]=0$ in $H_1(S)$.
Since $\phi$ is a relative embedding, 
we may assume that for each boundary component $\partial_i S$ of $S$, there exists a complete subgraph $K$ of $\Gamma$ 
 such that the curves and the arcs in $\HH$ intersecting with $\partial_i S$ are labeled by $V(K)$.
 
 Let $S'$ be a surface homeomorphic to $S$ by a homeomorphism $g:S\rightarrow S'$. Put
 $\partial_i S'=g(\partial_i S)$.
 We consider $D(S)$ as the union of $S,S'$ and the annuli $A_1,A_2\ldots,A_m$.
Here, $A_i$ is parametrized by $f_i : [-1,1]\times S^1\rightarrow A_i$, such that
 $f_i({-1}\times S^1)$ and $f_i({1}\times S^1)$ are glued to $\partial_i S$ and $\partial_i S'$, respectively.
 We will define a label-reading pair $(\HH',\lambda')$ on $D(S)$, which restricts to $(\HH,\lambda)$ on $S$.
To do this, we will write $\HH' = \HH_1' \cup \HH_2' \cup \HH_3'$  as follows.
 
$\HH_1'$ will be the collection of the simple closed curves $\gamma$ and $g(\gamma)$, for all simple closed curves $\gamma\in\HH$.
Here, we let $\gamma$ and $g(\gamma)$ in $\HH_1'$ inherit the label and the transverse orientation of $\gamma\in\HH$.

To define $\HH_2'$, let $\gamma\in\HH$ be a properly embedded arc, joinining $\partial_i S$ and $\partial_j S$.
Let $f_i(-1\times p_i)$ and $f_j(-1\times p_j)$ be the intersection of $\gamma$ with $\partial_i S$ and $\partial_j S$, respectively.
There exists a simple closed curve $\tilde\gamma$ on $D(S)$ obtained by taking a concatenation of $\gamma,f_j([-1,1]\times p_j),g(\gamma^{-1})$ and 
$f_i([-1,1]\times p_i)$. Again, we let 
$\tilde\gamma$ 
inherit the
label and the transverse orientation of $\gamma$, and define $\HH_2'$ to be the union of all such simple closed curves $\tilde\gamma$ on $D(S)$, where the
union is taken over all properly embedded arcs
$\gamma\in\HH$.

Now we define $\HH_3'$ as follows. 
Consider any boundary component $\partial_i S$,
and let $\beta_1,\beta_2,\ldots,\beta_s\subseteq S$ be the properly embedded arcs in $\HH$
intersecting with $\partial_i S$.
There exists
a (possibly non-unique) maximal complete subgraph $K\le\Gamma$ such that 
$\lambda(\beta_j)\in V(K)$ for all $j$.
We choose  disjoint essential simple closed curves $\alpha_1,\ldots,\alpha_s$ in the interior of  $A_i$, 
and let 
$\lambda'(\alpha_j)=v_{K,\lambda(\beta_j)}\in V(\Gamma^*)$, for each $j$.
 Moreover, we let the transverse orientation of $\alpha_j$ be from $f_i(-1\times S^1)$ to $f_i(1\times S^1)$, 
 if the transverse orientation of $\beta_j$ coincides with the orientation of $\partial_i S$, and be the opposite otherwise (Figure~\ref{fig:twist}\;(a)).
Let $\HH_3'$ be the union of
all such $\alpha_j$'s, where the union is taken over all
the boundary components $\partial_1 S,\ldots,\partial_m S$. In this way, we have defined a set of transversely oriented curves
and arcs 
$\HH'=\HH_1'\cup\HH_2'\cup\HH_3'$ and a labeling $\lambda':\HH'\rightarrow V(\Gamma^*)$. 

Let $\phi':\pi_1(D(S))\rightarrow A(\Gamma^*)$
be the label-reading map with respect to $(\HH',\lambda')$. Fix any $n>0$. Define $p_n:A(\Gamma^*)\rightarrow A(\Gamma)$
by 
$p_n(a) =a$ for $a\in V(\Gamma)$,
and
$p_n(v_{K,u}) = u^n$ for $K\in \KK(\Gamma)$ and $u\in V(K)$. 
Let $T_i$ be the Dehn Twist of $D(S)$ along $\partial_i S$, and $T=T_1\circ T_2\circ\cdots \circ T_m$.

\begin{claim*}
The  following diagram commutes up to equivalence.
\[\xymatrix{	\pi_1(D(S))\ar[d]^{(q\circ T^n)_*}\ar[r]^{\phi'}
 	& A(\Gamma^*)\ar[d]^{p_n} \\
	\pi_1(S)\ar[r]^{\phi} & A(\Gamma) 
	}
\]
\end{claim*}

Note that $\phi\circ q_*$ is the label-reading map with respect to the pair
$(\HH_1'\cup\HH_2',\lambda')$.
Similarly, $\phi\circ (q\circ T^n)_*$
is
the label-reading map with respect to the pair consisting of the set
$T^{-n}(\HH_1'\cup\HH_2')$, and the labeling map $\lambda'\circ T^n$.

Let $\alpha\in\HH_3'$ be any simple closed curve inside an annulus, say $A_i$.
Write $\lambda'(\alpha) = v_{K,u}$ for some $K\in\KK(\Gamma)$ and $u\in V(K)$.
Consider $n$ copies of disjoint essential simple closed curves
$\tilde\alpha_1,\tilde\alpha_2,\cdots,\tilde\alpha_n\subseteq A_i$,
with the same transverse orientation as $\alpha$.
Label $\tilde\alpha_1,\tilde\alpha_2,\cdots,\tilde\alpha_n$ by $u$.
Define
$\CCC_n$ to be the union of all such $\tilde\alpha_j$'s, where
the union is taken over all the simple closed curves $\alpha\in\HH_3'$.
Then
$\HH_1'\cup\HH_2'\cup\CCC_n$ with the transverse orientations and the labeling defined so far
determines a label-reading map  $\psi:\pi_1(D(S))\rightarrow A(\Gamma)$. Note that two curves or arcs of the same
label are not necessarily disjoint in this label-reading pair (Figure~\ref{fig:twist}\;(b)). 
From the construction, $\psi = p_n\circ\phi'$.

From Lemma~\ref{lem:equiv} (3), one immediately sees that
$\psi$ is equivalent to the label-reading map with respect to $(T^{-n}(\HH_1'\cup\HH_2'),\lambda'\circ T^n)$ (Figure~\ref{fig:twist}\;(c)).
Therefore, 
$p_n\circ \phi'(x)$ is equivalent to
$\phi\circ (q\circ T^n)_*(x)$. The claim is proved.

Now suppose $x$ is a non-trivial element in $\pi_1(D(S))$.
By Lemma~\ref{lem:evinj}, there exists $n\ge0$ such that $(q\circ T^n)_*(x) = (q\circ T_1^n\circ T_2^n\circ\cdots T_m^n)_*(x)\ne1$.
The injectivity of $\phi$ and the commutativity of the diagram above imply that $\phi'(x)$ is non-trivial. 
Hence, $\phi'$ is injective, and so, $\Gamma^*\not\in\NNN$.
 \end{proof}

\begin{figure}[htb!] 
\subfigure[$\HH' = \HH'_1\cup\HH'_2\cup\HH'_3$]{
 \labellist 
\small\hair 2pt 
\pinlabel {$\scriptstyle A_i$} [l] at 53 88
\pinlabel {$\leftarrow$} [l] at 6 8
\pinlabel {$\scriptstyle b$} [l] at 10 1
\pinlabel {$\scriptstyle a$} [l] at 22 43
\pinlabel {$\scriptstyle v_{K,a}$} [l] at 44 63
\pinlabel {$\scriptstyle v_{K,b}$} [l] at 79 62
\pinlabel {$\rightarrow$} [l] at 13 37
\pinlabel {$\leftarrow$} [l] at 33 61
\pinlabel {$\rightarrow$} [l] at 68 61
\pinlabel {$\scriptstyle\partial_i S$} [lb] at 28 20
\pinlabel {$\scriptstyle\partial_i S'$} [rb] at 87 20
\endlabellist 
\centering 
\includegraphics{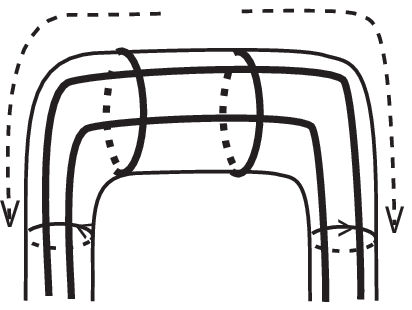}
}
\subfigure[$\HH'_1\cup\HH'_2\cup\CCC_n$]{
 \labellist 
\small\hair 2pt 
\pinlabel {$\leftarrow$} [l] at 7 10
\pinlabel {$\scriptstyle b$} [l] at 10 1
\pinlabel {$\scriptstyle a$} [l] at 21 41
\pinlabel {$\scriptstyle a$} [lb] at 26 80
\pinlabel {$\scriptstyle a$} [lb] at 43 80
\pinlabel {$\scriptstyle b$} [lb] at 60 80
\pinlabel {$\scriptstyle b$} [lb] at 79 80
\pinlabel {$\rightarrow$} [l] at 13 34
\pinlabel {$\leftarrow$} [l] at 30 61
\pinlabel {$\leftarrow$} [l] at 46 61
\pinlabel {$\rightarrow$} [l] at 66 61
\pinlabel {$\rightarrow$} [l] at 83 61
\pinlabel {$\scriptstyle\partial_i S$} [lb] at 28 20
\pinlabel {$\scriptstyle\partial_i S'$} [rb] at 87 20
\endlabellist 
\centering 
\includegraphics{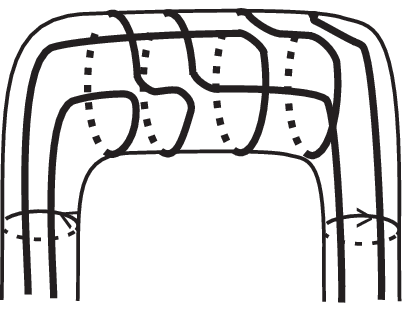}
}
\subfigure[$T^{-n}(\HH_1'\cup\HH_2')$]{
 \labellist 
\small\hair 2pt 
\pinlabel {$\scriptstyle b$} [l] at 7  2
\pinlabel {$\scriptstyle a$} [l] at 19 42
\pinlabel {$\rightarrow$} [l] at 10 37
\pinlabel {$\leftarrow$} [l] at 4 10
\pinlabel {$\scriptstyle\partial_i S$} [lb] at 27 23
\pinlabel {$\scriptstyle\partial_i S'$} [rb] at 87 20
\endlabellist 
\centering 
\includegraphics{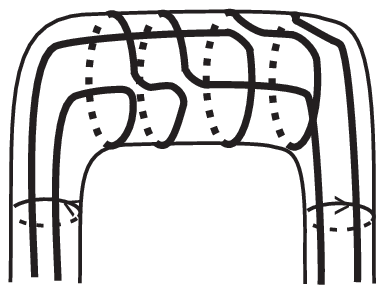}
}
\caption[Defining a label-reading pair on the double of a surface]{Defining a label-reading pair $(\HH',\lambda')$ on $D(S)$. 
(b) and (c) show equivalent label-reading pairs.  \label{fig:twist}}
\end{figure}

\begin{rem}\label{rem:promotion}
Lemma~\ref{lem:evinj} is interesting in its own right, and naturally used to prove Lemma~\ref{lem:promotion}.
One can also see that $\phi'$ is injective in the above proof by a method from~\cite{CSS2008}, 
rather than resorting to Lemma~\ref{lem:evinj}, as is briefly discussed below.
The labels of the curves inside 
two different annuli $A_i$ and $A_j$ are either simplicial vertices joined to the same maximal complete subgraph of $\Gamma$,
or disjoint and independent.
Now suppose $\phi'_*[\alpha]=1$, for some $[\alpha]\in\pi_1(D(S))\setminus\{1\}$.
By the solution to the word problem for right-angled Artin groups, $\alpha$ has an essential subarc $\beta$ (in $S$ or $S'$)
satisfying the following: $\beta$ intersects  
boundary components $\partial_i S$ and $\partial_j S$, such that 
the curves inside $A_i$ and $A_j$ are labeled by simplicial vertices joined to the same maximal complete subgraph $K$ of $\Gamma$
and
the label-reading by $(\HH,\lambda)$ of $\beta$ is in $A(K)$ (note that $i=j$ may occur).
This implies that any arc intersecting with $\partial_i S$ or $\partial_j S$ is labeled by a vertex of $K$.
Orient $\beta$ from $\partial_i S$ to $\partial_j S$. Then the label-reading by $(\HH,\lambda)$ of 
$\beta\cdot\partial_j S\cdot\beta^{-1}$ and $\partial_i S$
are both in $A(K)$, and so, commute. This will contradict to the injectivity of $\phi$.
\end{rem}

\begin{thm}\label{thm:completeamalgam}
$\NNN'$ is closed under complete graph amalgamation.
\end{thm}

\begin{proof}
Let $\Gamma,\Gamma_1$ and $\Gamma_2$ be graphs, such that 
$\Gamma$ is a complete graph amalgamation of $\Gamma_1$ and $\Gamma_2$.
We will show that  $\Gamma\in\NNN'$ if and only if $\Gamma_1,\Gamma_2\in\NNN'$.

($\Rightarrow$) Obvious, since $\Gamma_i\le\Gamma$. 

($\Leftarrow$) Assume $\Gamma\not\in\NNN'$.
By Lemma~\ref{lem:promotion}, $\Gamma^\ast\not\in \NNN$. 
$\Gamma^*$ is a complete graph amalgamation of induced subgraphs $\Gamma_1'\ge\Gamma_1$ and $\Gamma_2'\ge\Gamma_2$, as in Lemma~\ref{lem:simplicial_extension}. 
By Lemma~\ref{lem:nni}, we may assume $\Gamma_1'\not\in\NNN'$. 
Since
$\Gamma_1'$ can be obtained by adding independent simplicial vertices to $\Gamma_1$,
Lemma~\ref{lem:simplicialv} implies
 $\Gamma_1\not\in\NNN'$
 \end{proof}

\begin{cor} \label{cor:chordal}
Any chordal graph is in $\NNN'$.
\end{cor}

\begin{proof}
Note that each complete graph is in $\NNN'$, since free abelian groups do not contain closed hyperbolic
surface group or non-abelian free groups.  
For each chordal graph $\Gamma$, either $\Gamma$ is complete or $\Gamma$ can be written as a complete graph amalgamation
$\Gamma = \Gamma_1 \cup \Gamma_2$ of proper induced subgraphs $\Gamma_1$ and 
$\Gamma_2$~\cite{dirac1961}.
By Theorem~\ref{thm:completeamalgam}, an inductive argument shows that $\Gamma$ is in $\NNN'$.
 \end{proof}

 In particular, $A(\Gamma)$ does not contain a closed hyperbolic surface group if $\Gamma$ is chordal.
The condition that the underlying graph $\Gamma$ is chordal
is equivalent to two important group theoretic properties on $A(\Gamma)$. Namely,
$\Gamma$ is chordal, if and only if 
  $A(\Gamma)$ is coherent~\cite{droms1987a}, if and only if
  $A(\Gamma)$ has a free commutator subgroup~\cite{SDS1989}.
 
\section{Normalized Label-Reading Pairs} \label{sec:labelreading}
In this section, we let $\Gamma$ be a graph and $S$ be a compact surface.
For a given label-reading pair on $S$ with the underlying graph $\Gamma$, 
we will consider a simplification (called, {\em normalization})
of the label-reading pair, without changing the equivalence class of
the induced label-reading map (Definition~\ref{defn:norm}). 
Lemma~\ref{lem:norm1} and~\ref{lem:norm2} will be crucially used in Section~\ref{sec:bisimp}.

\bd[regular label-reading pair]\label{defn:regular}
A  label-reading pair $(\HH,\lambda)$ on $S$ with the underlying graph $\Gamma$ 
is called {\em regular}, if the following are satisfied.
\enumir
\be
\item
The induced label-reading map $\phi:\pi_1(S)\rightarrow A(\Gamma)$ is injective.
\item
The curves and the arcs in $\HH$ are 
neither null-homotopic nor homotopic into the boundary.
\item
Any curves and arcs in $\HH$ are minimally intersecting. This means that for any $\alpha\ne\beta$ in $\HH$, 
$| \alpha\cap\beta | = i(\alpha,\beta)$.
\item
Two curves or arcs of the same label do not intersect.
\item
For each boundary component 
 $\partial_i S$,
 there exists a complete subgraph $K\le \Gamma$
 such that any arc $\alpha$ intersecting with $\partial_i S$ satisfies $\lambda(\alpha)\in V(K)$.
 \ee
\enumia
\ed

From Proposition~\ref{prop:lrexists} and~\ref{lem:equiv}, 
any relative embedding $\phi:\pi_1(S)\rightarrow A(\Gamma)$ 
is induced by a regular label-reading pair $(\HH,\lambda)$, possibly after a conjugation in $A(\Gamma)$.
This is the first step to simplify a given label-reading pair. 

\bd[normalized label-reading pair] \label{defn:norm}
\enumia
\be
\item
Let $(\HH,\lambda)$ be a regular label-reading pair on a hyperbolic
surface $S$, 
and $\BBB$ be the set of properly embedded arcs in $\HH$.
Define the {\em complexity of $\HH$} to be the 4-tuple of the nonnegative intergers
\[
c(\HH,\lambda)=\left(|(\cup\BBB)\cap\partial S|,\sum_{a\in V(\Gamma)} |\lambda^{-1}(a) \big/ \sim|,
\sum_{a\in V(\Gamma)} |\lambda^{-1}(a)\cap\BBB \big/ \sim|,
\sum_{{\alpha,\beta \in\HH} \atop { \alpha\ne\beta}} |\alpha\cap\beta| \right)\]
where
 $\sim$ denotes the homotopy equivalence relation on $\HH$, also and on $\BBB$.
We denote the lexicographical ordering of 
the complexities by  $\preceq$.
\item
A regular label-reading pair 
$(\HH,\lambda)$ is {\em normalized} if for any other
regular label-reading pair $(\HH',\lambda')$ which is equivalent to $(\HH,\lambda)$,
$c(\HH,\lambda)\preceq c(\HH',\lambda')$.
\ee
\end{defn}

In the above definition, $ |(\cup\BBB)\cap\partial S| $ denotes the number of intersection points between $\partial S$ 
and the arcs in $\BBB$. This means, $ | (\cup\BBB)\cap\partial S| = | (\cup\HH) \cap \partial S|$ is the number of the endpoints of 
arcs in $\BBB$. 
 It is obvious that any regular label-reading pair is equivalent to a normalized one. 
We start with a simple observation on normalized label-reading pairs.

\begin{lem}[Normalization I] \label{lem:norm1}
Let  $(\HH,\lambda)$ be a normalized label-reading pair on a compact hyperbolic surface $S$ with the underlying graph $\Gamma$.
$\BBB$ denotes the set of properly embedded arcs in $\HH$.
 If   $\alpha,\beta\in \BBB$ have the same label and intersect with
the same boundary component $\partial_i S$, then the transverse orientation of
$\alpha$ and that of $\beta$ induce the same orientation on  $\partial_i S$ at their intersections with
$\partial_i S$.
In particular, each properly embedded arc in $\HH$ intersects with two distinct boundary components of $S$.
\end{lem}

\begin{proof} 
Let $a= \lambda(\alpha) = \lambda(\beta)$.
Suppose the transverse orientations of $\alpha$ and $\beta$ do not induce the 
same orientation on $\partial_i S$ at their intersection points $\{P_\alpha,P_\beta\}$.
By choosing a nearest one among such pairs of intersection points on $\partial_i S$,
we may assume a component of $\partial_i S \setminus \{P_\alpha,P_\beta\}$
does not intersect with any $a$--arc (Figure~\ref{fig:norm1}\;(a)). 
By Lemma~\ref{lem:equiv},  one can reduce 
 $|(\cup\BBB)\cap\partial S|$ by 2 without changing the equivalence class of $(\HH,\lambda)$, if
 one replaces $\alpha$ and $\beta$
 by another curve or arc $\alpha'$ as in Figure~\ref{fig:norm1}\;(b).
 Note that this new label-reading pair can be further simplified to become regular, again by Lemma~\ref{lem:equiv}.
  \end{proof}

\begin{rem}
During the proof of Lemma~\ref{lem:norm1}, one might have increased the number of homotopy classes of simple closed curves in $\lambda^{-1}(a)$, when $\alpha$ and $\beta$ are replaced by $\alpha'$. But the proof is still valid, since we are considering the lexicographical ordering of the complexity.
\end{rem}

\begin{figure}[htb!] 
\centering
\subfigure[]{
\labellist 
\small\hair 2pt 
\pinlabel {$\scriptstyle \partial_i S$}   at 21 22
\pinlabel {$\scriptstyle S$} [rb] at 5 41
\pinlabel {$\scriptstyle P_\alpha$} [b] at 35 42
\pinlabel {$\scriptstyle P_\beta$} [t] at 45 11
\pinlabel {$\scriptstyle a$} [rb] at 63 49
\pinlabel {$\searrow$} [tl] at 61 51
\pinlabel {$\scriptstyle b$} [b] at 88 36 
\pinlabel {$\uparrow$} [t] at 83 42
\pinlabel {$\scriptstyle c$} [b] at 71 25
\pinlabel {$\uparrow$} [t] at 66 30
\pinlabel {$\scriptstyle \alpha$} [lb] at 85 56
\pinlabel {$\scriptstyle \beta$} [t] at 87 1
\pinlabel {$\scriptstyle a$} [tr] at 63 6
\pinlabel {$\nearrow$} [bl] at 61 4
\endlabellist
\centering
\includegraphics{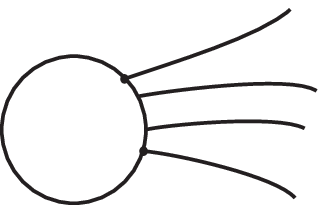}
}
\hspace{0.8in}
\subfigure[]{
\labellist 
\small\hair 2pt 
\pinlabel {$\scriptstyle \partial_i S$}   at 21 22
\pinlabel {$\scriptstyle S$} [rb] at 4 40
\pinlabel {$\scriptstyle a$} [rb] at 63 49
\pinlabel {$\searrow$} [tl] at 61 51
\pinlabel {$\scriptstyle b$} [b] at 88 36 
\pinlabel {$\uparrow$} [t] at 83 40
\pinlabel {$\scriptstyle c$} [b] at 71 25
\pinlabel {$\uparrow$} [t] at 66 29
\pinlabel {$\scriptstyle \alpha'$} [lb] at 85 56
\pinlabel {$\scriptstyle a$} [tr] at 63 6
\pinlabel {$\nearrow$} [bl] at 61 4
\endlabellist
\centering
\includegraphics{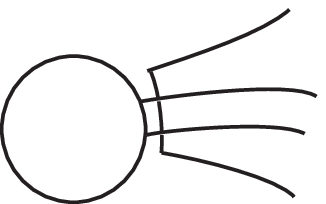}
}
\caption{Reducing  complexity.
In (a), the labels $b$ and $c$ are adjacent to $a$ in $\Gamma$ by the regularity of $(\HH,\lambda)$. 
Hence in (b), the intersections of $\alpha'$  with $b$--and $c$--arcs are allowed to occur.
Consequently, $\alpha$ and $\beta$ can be replaced by $\alpha'$ without changing the equivalence class of the label-reading pair.}
\label{fig:norm1}
\end{figure}

To state properties of normalized label-reading pairs, 
it will be convenient to define certain terms regarding a set of disjoint properly embedded arcs on $S$, as follows. Let $I$ denote the unit interval $[0,1]$.

\enumiir
\bd[Strips and Channels] \label{defn:stripchan}
Let $\AAA$ be a set of disjoint properly embedded arcs on a compact surface $S$.
\be
\item
We choose an embedding $\eta_{\alpha}:I\times[-1,1]\rightarrow S$
for each arc $\alpha\in\AAA$, such that the following conditions hold.
\be
\item
$\eta_{\alpha}(I\times s)$ is 
a properly embedded arc for each $s\in[-1,1]$.
\item
$\alpha\subseteq \eta_\alpha(I\times(-1,1))$
\item
If $\alpha\sim\alpha'\in\AAA$, then
$\eta_\alpha = \eta_{\alpha'}$.
\item
If $\alpha\not\sim\alpha'\in\AAA$, then
the image of $\eta_\alpha$ and that of $\eta_{\alpha'}$ are disjoint.
\ee
We call $\{\eta_\alpha\;|\; \alpha\in\AAA\}$
as a {\em set of strips} for $\AAA$.
For convenience, the image of $\eta_\alpha$ is also denoted by $\eta_\alpha$,
when there is no danger of confusion.
\item
A {\em channel} is a connected component of 
$(\cup_{\alpha\in\AAA}\thinspace\eta_\alpha)\cup\partial S$.
For $\alpha\in\AAA$,
we denote the unique channel containing $\alpha$ by $\ch(\alpha)$.
An {\em induced simple closed curve} of $\alpha$ is
a boundary component $\hat\alpha$ of the closure of $S\setminus\ch(\alpha)$ such that
$\hat\alpha\cap\eta_\alpha\ne\varnothing$ and
$\hat\alpha\not\subseteq\partial S$.
Note that there exist at most two induced simple closed curves of $\alpha$ for each $\alpha\in\AAA$.
\item
An arc $\alpha\in\AAA$ is {\em one-sided with respect to $\AAA$},
if $\eta_\alpha(I\times\{-1\})$ and $\eta_\alpha(I\times\{1\})$ 
are contained in the same induced simple closed curve  (see Figure~\ref{fig:1sided}).
\ee
\ed

\begin{figure}[bh!] 
\centering
\subfigure[$\alpha$ is not one-sided]{
\labellist 
\small\hair 2pt 
\pinlabel {${}_{S}$} [l] at 5 100
\pinlabel {${}_{\alpha}$} [l] at 40 62
\pinlabel {${}_{\hat\alpha}$} [l] at 57 83
\pinlabel {${}_{\hat\alpha'}$} [l] at 60 27
\pinlabel {${}_{\partial_1 S'}$} [b] at 16 50
\pinlabel {${}_{\partial_2 S'}$} [b] at 67 50 
\pinlabel {${}_{\partial_3 S'}$} [b] at 112 80 
\pinlabel {${}_{\partial_4 S'}$} [b] at  110 23
\endlabellist
\centering
\includegraphics{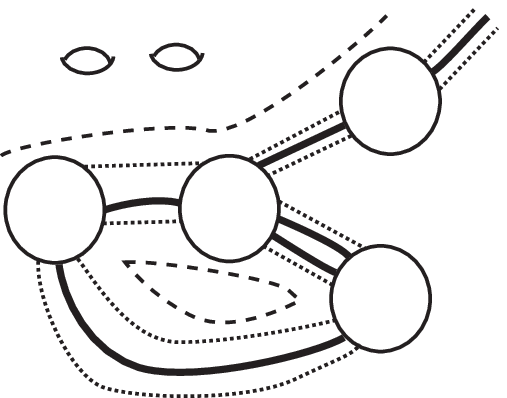}
}
\hspace{0.2in}
\subfigure[$\alpha$ is one-sided]{
\labellist 
\small\hair 2pt 
\pinlabel {${}_{S}$} [l] at 5 95
\pinlabel {${}_{\alpha}$} [l] at 35 55
\pinlabel {${}_{\hat\alpha}$} [l] at 56 86
\pinlabel {${}_{\partial_1 S'}$} [b] at 16 44
\pinlabel {${}_{\partial_2 S'}$} [b] at 67 43 
\pinlabel {${}_{\partial_3 S'}$} [b] at 112 75 
\pinlabel {${}_{\partial_4 S'}$} [b] at  110 18
\endlabellist
\centering
\includegraphics{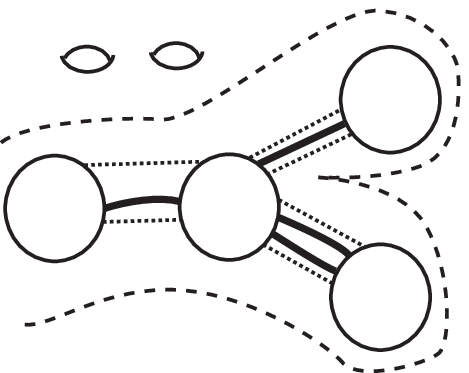}
}
\caption{Strips and channels of a set $\AAA$ of disjoint properly embedded arcs. 
Curves in $\AAA$ are drawn bold. The dotted arcs 
 bound strips, along with some intervals on $\partial S$. 
 $\hat\alpha$ and $\hat\alpha'$ are induced simple closed curves. They do intersect $\partial S$
 and the boundaries of strips,
 but for convenience of drawing, the figures show curves parallel to $\hat\alpha$ and $\hat\alpha'$, in the interior
 of the surface.
In (a), $\hat\alpha$ and $\hat\alpha'$ are distinct induced simple closed curves of $\alpha\in\AAA$.
In (b), $\hat\alpha$ is the unique induced simple closed curve of $\alpha$.
\label{fig:1sided}}
\end{figure}

\begin{rem}\label{rem:induced}
Let $\AAA$ be a set of disjoint properly embedded arcs on a compact surface $S$, and $\alpha\in\AAA$. 
\be
\item
Any induced simple closed curve $\hat\alpha$ can be written as a concatenation of paths
\[ \hat\alpha = \alpha'_1 \cdot \delta_1 \cdot \alpha'_2 \cdot \delta_2 \cdot\alpha'_3 \cdots \alpha'_r\]
such that
\enumir
\be
\item
$\alpha_1' \sim \alpha$,
\item
$\alpha_i'\sim\alpha_i$ for some $\alpha_i\in\AAA$, and $\alpha_i'$ is an interval on the boundary of the strip of $\alpha_i$,
\item
$\delta_i$ is an interval on a boundary component of $S$ that intersects with $\alpha_i'$ and $\alpha_{i+1}'$.
\ee
\enumia
In particular, an induced simple closed curve consists of subarcs which lie on $\partial S$
and the boundaries of strips.
Moreover, if $\alpha$ is not one-sided, 
 then the transverse orientation of $\alpha$  uniquely determines a transverse orientation of 
 $\hat\alpha$ that respects the homotopy $\alpha_1'\sim\alpha$.
\item
For a sufficiently small closed regular neighborhood $N$ of $\ch(\alpha)$,
there exist disjoint annuli $A_1,A_2,\ldots,A_r$ in the closure of $S\setminus\ch(\alpha)$,
such that $N = \ch(\alpha)\cup A_1\cup A_2\cup\cdots\cup A_r$.
The intersection of each $A_i$ with $\ch(\alpha)$ is an induced simple closed curve,
and conversely, any induced simple closed curve intersecting with $\ch(\alpha)$
 is a boundary component of some $A_i$.
\item
Let $(\HH,\lambda)$ be a label-reading pair on $S$ with the underlying graph $\Gamma$.
Consider any arc $\alpha\in\HH$, and let $a = \lambda(\alpha)$.
Denote the set of $a$--arcs by $\AAA_a$.
We may assume that the arcs in $\AAA_a$ are disjoint (Lemma~\ref{lem:equiv} (3)).
Then the {\em strip, the channel, and the induced simple closed curves} of $\alpha$
are defined to be those of $\alpha$ with respect to the set $\AAA_a$.
Furthermore, $\alpha$ is said to be {\em one-sided} if it is one-sided with respect to $\AAA_a$.
\ee
\end{rem}

\begin{lem}[Normalization II] \label{lem:norm2}
Let  $(\HH,\lambda)$ be a normalized label-reading pair on a compact hyperbolic surface $S$ with the underlying graph $\Gamma$.
Then  each arc $\alpha$ in $\HH$ is one-sided (see Remark~\ref{rem:induced} (3)).
\end{lem}

\begin{proof} 
Let $a=\lambda(\alpha)$.
For abbreviation, we simply let $\AAA$ denote the set of $a$--arcs.
Suppose there exists an arc $\alpha\in\AAA$, 
which is not one-sided in $\AAA$.
Let $\hat\alpha$ be  one of the two induced simple closed curves  of $\alpha$ 
with respect to $\AAA$.
Write
$\hat\alpha = \alpha'_1 \cdot \delta_1 \cdot \alpha'_2 \cdot \delta_2\cdots
\alpha'_r\cdot \delta_r$,
where  $\delta_i\subseteq\partial S$, and
$\alpha'_i$ is a properly embedded arc homotopic to an $a$--arc $\alpha_i\in\AAA$, as in
Remark~\ref{rem:induced} (1). Here, $\alpha_1'\sim\alpha$.
 The transverse orientation of $\hat\alpha$ is given by that of $\alpha$. 

First,  consider the case when no other curve in $\AAA$ is  homotopic to $\alpha$.
Choose an embedding $g:S^1\times I\rightarrow \overline{S\setminus\ch(\alpha)}$ such that
$g(S^1\times\{0\})=\hat\alpha$, as in Remark~\ref{rem:induced} (2).
Put $\beta=g(S^1\times\{\frac12\})$
and $\gamma = g(S^1\times\{1\})$. One may assume $\gamma$ is sufficiently close 
to $\hat\alpha$, so that if any $\gamma'\in\HH$ intersects with $\gamma$, then $\gamma'$ also
intersects with $\hat\alpha$.
Let $\gamma$ have the transverse orientation which the homotopy $\gamma\sim\hat\alpha$ respects, 
and give $\beta$ the opposite orientation (Figure~\ref{fig:norm2}\;(a)). Label $\beta$ and $\gamma$
by $a$ and add them to $\HH$, 
resulting in a new label-reading pair $(\HH_1,\lambda_1)$. Note that $(\HH_1,\lambda_1)$
is equivalent to $(\HH,\lambda)$, since the homotopic curves $\beta$ and $\gamma$ have the same label
and the opposite transverse orientations.

Note that the transverse orientations of $\alpha_2,\alpha_3,\alpha_4,\alpha_5,\ldots$ 
are completely determined by that of $\alpha$, according to Lemma~\ref{lem:norm1}.
Obtain another label-reading pair $(\HH_2,\lambda_2)$ from $(\HH_1,\lambda_1)$
by  removing $\beta,\alpha,\alpha_3,\alpha_5,\ldots$ and adding $a$--arcs homotopic to
$\alpha_2,\alpha_4,\ldots$. Here, newly added arcs will have 
the transverse orientations respecting the homotopies to $\alpha_2,\alpha_4,\ldots$, as in
 Figure~\ref{fig:norm2}\;(c). 
One sees that $(\HH_2,\lambda_2)$ is equivalent to $(\HH_1,\lambda_1)$ 
by successive applications of Lemma~\ref{lem:equiv} (3). 
Figure~\ref{fig:norm2}\;(b) illustrates an intermediate step between 
$(\HH_1,\lambda_1)$ and $(\HH_2,\lambda_2)$.    

Since $\alpha$ is not one-sided, $\alpha_i\not\sim \alpha$ for each $i>1$. 
We have assumed that $\alpha$ is the only arc in its homotopy class, contained in $\AAA$.
Hence, $(\HH_2,\lambda_2)$ does not contain any $a$--arc homotopic to $\alpha$.
This means,
$\HH_2$ has a strictly smaller number of homotopy classes of properly embedded arcs,
than $\HH$ does.
$\HH_2$ has the same number of, or one more, homotopy classes of simple closed curves
than  $\HH$ does, according to whether there exists any $a$--curve homotopic to $\gamma$
in $\HH$ or not.
This implies, $|\lambda_2^{-1}(a)\big/\sim|\le|\lambda^{-1}(a)\big/\sim|$.
Moreover,   $| \HH \cap   \partial S| = | \HH_2  \cap   \partial S|$. Hence, 
$c(\HH_2,\lambda_2)\prec c(\HH,\lambda)$. This is a contradiction to the minimality of $c(\HH,\lambda)$.

In the case when there exist $l>1$ properly embedded arcs in $\HH$ homotopic to
$\alpha$, fix a small annulus $A\subseteq \overline{S\setminus\ch(\alpha)}$, of which $\hat\alpha$ is a boundary component.
Consider a set of disjoint, transversely oriented  simple closed curves
$\beta_1,\beta_2,\ldots,\beta_l,\gamma_1,\gamma_2,\ldots,\gamma_l$ 
contained in $A$ with this order.
Here, we let
$\beta_1,\beta_2,\ldots,\beta_l$
 have the opposite transverse orientations to
that of $\hat\alpha$, and $\gamma_1,\gamma_2,\ldots,\gamma_l$ 
have the transverse orientations coinciding with that of $\hat\alpha$.
By letting $\HH_1 = \HH\cup\{\beta_1,\beta_2,\ldots,\beta_l,\gamma_1,\gamma_2,\ldots,\gamma_l\}$,
the same argument implies that $c(\HH,\lambda)$ is not minimal.
 \end{proof}

\begin{figure}[htb!] 
\centering
\subfigure[$(\HH_1,\lambda_1)$]{
\labellist 
\small\hair 2pt 
\pinlabel {${}{\downarrow}$} [b] at 57 85
\pinlabel {${}{\swarrow}$} [tr] at 32 8
\pinlabel {${}{\uparrow}$}      at 57 79
\pinlabel {${}{\downarrow}$} [b] at 57 58
\pinlabel {${}_{\alpha}$} [lb] at 60 93
\pinlabel {${}_{\alpha_2}$} [rb] at 22 84
\pinlabel {${}{\nwarrow}$} [br] at 18 71
\pinlabel {${}_{\alpha_3}$} [tr] at 2 40
\pinlabel {${}{\nearrow}$} [tr] at 12 38
\pinlabel {${}_{\alpha_r}$} [lb] at 95 65
\pinlabel {${}{\nearrow}$} [b] at 91 68
\pinlabel {${}_{\beta}$} [l] at 86 35
\pinlabel {${}_{\gamma}$} [r] at 69 35
\pinlabel {${}_{\partial S}$}      at 77 86
\pinlabel {${}_{\partial S}$}      at 35 86
\pinlabel {${}_{\partial S}$}      at 8 56
\pinlabel {${}_{\partial S}$}      at 15 16
\endlabellist
\centering
\includegraphics{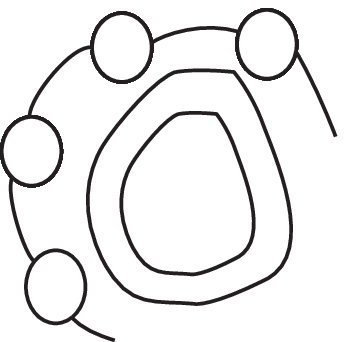}
}
\hfill
\subfigure[a reduction of $(\HH_1,\lambda_1)$]{
\labellist 
\small\hair 2pt 
\pinlabel {${}{\swarrow}$} [tr] at 32 8
\pinlabel {${}{\downarrow}$} [b] at 57 58
\pinlabel {${}_{\alpha_2}$} [rb] at 22 84
\pinlabel {${}{\nwarrow}$} [br] at 18 71
\pinlabel {${}{\nwarrow}$} [tl] at 26 70
\pinlabel {${}_{\alpha_3}$} [tr] at 2 40
\pinlabel {${}{\nearrow}$} [tr] at 12 38
\pinlabel {${}{\swarrow}$} [tr] at 32 40
\pinlabel {${}_{\alpha_r}$} [lb] at 95 65
\pinlabel {${}{\nearrow}$} [b] at 91 68
\pinlabel {${}_{\gamma}$} [r] at 69 35
\pinlabel {${}_{\partial S}$}      at 77 87
\pinlabel {${}_{\partial S}$}      at 34 86
\pinlabel {${}_{\partial S}$}      at 8 56
\pinlabel {${}_{\partial S}$}      at 15 16
\endlabellist
\centering
\includegraphics{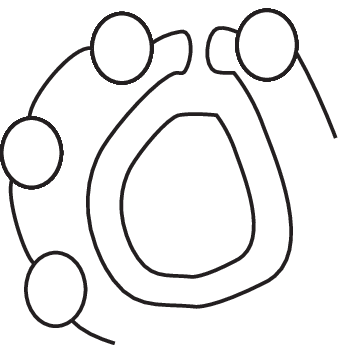}
}
\hfill
\subfigure[$(\HH_2,\lambda_2)$]{
\labellist 
\small\hair 2pt 
\pinlabel {${}{\downarrow}$} [b] at 57 59
\pinlabel {${}_{\alpha_2}$} [rb] at 22 84
\pinlabel {${}{\nwarrow}$} [br] at 18 71
\pinlabel {${}{\nwarrow}$} [br] at 29 66
\pinlabel {${}_{\alpha_r}$} [lb] at 95 65
\pinlabel {${}{\nearrow}$} [b] at 91 68
\pinlabel {${}{\nearrow}$}      at 87 61
\pinlabel {${}{\swarrow}$} [tr] at 32 8
\pinlabel {${}{\swarrow}$} [bl] at 33 1
\pinlabel {${}_{\gamma}$} [r] at 69 35
\pinlabel {${}_{\partial S}$}      at 77 87
\pinlabel {${}_{\partial S}$}      at 34 86
\pinlabel {${}_{\partial S}$}      at 8 56
\pinlabel {${}_{\partial S}$}      at 15 16
\endlabellist
\centering
\includegraphics{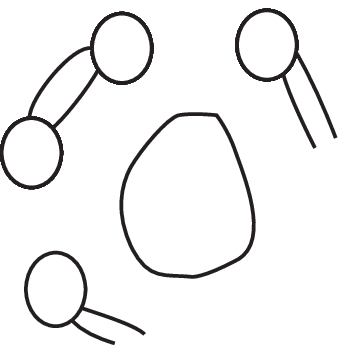}
}
\caption{Proof of Lemma~\ref{lem:norm2}. All the curves and arcs drawn here are labeled by $a$.  \label{fig:norm2}}
\label{}
\end{figure}

Now we state a lemma regarding a set of disjoint properly embedded arcs on a compact surface, 
such that each arc is one-sided. In view of Lemma~\ref{lem:norm2}, this result will be 
applied to the study of normalized label-reading pairs in the next section. 
Note, the conclusion of Lemma~\ref{lem:1sided} is not true without the hypothesis that
each arc is one-sided.

\begin{lem} \label{lem:1sided}
Let $\AAA$ be a set of disjoint properly embedded arcs on a surface $S$,
such that each arc in $\AAA$ is one-sided. 
Denote the union of the boundary components of $S$ 
intersecting with arcs in $\AAA$ by $\partial^* S$.
Fix  $\alpha\in\AAA$, and let $\hat\alpha$ be the unique induced simple closed curve of $\alpha$.
For a sufficiently small closed regular neighborhood $N$ of $\ch(\alpha)$, the following hold.
\be
\item
$N$ has a unique boundary component, say $\hat\alpha'$, that is not a boundary component of $S$. 
$\hat\alpha'$ separates $S$, and
$\hat\alpha'\sim\hat\alpha$.
\item
A properly embedded arc or a closed curve, not intersecting with 
$(\cup\AAA)\cup\partial^* S$, can be homotoped into $S\setminus N$.
\item
If we further assume that $\hat\alpha$ is null-homotopic, then $\partial^* S = \partial S$ and
any essential closed curve on $S$ intersects with $\cup\AAA$.
\ee
\end{lem}

\begin{proof} (1) 
We use the notations in Definition~\ref{defn:stripchan}.
We will say that a boundary component of $S$ or a strip of $\AAA$
is {\em good} if it intersects with $\hat\alpha$ (Remark~\ref{rem:induced} (1)).

\begin{claim2}\label{ccl1}
If a strip is good, then so is any boundary component of $S$ intersecting with that strip.
\end{claim2}

If a strip $\eta_\beta:I\times [-1,1]\rightarrow S$ is good for some $\beta\in\AAA$,
then $\eta_\beta(I\times-1)$ or 
$\eta_\beta(I\times1)$ is contained in $\hat\alpha$.
Since $\beta$ is one-sided,
$\eta_\beta(I\times\{-1,1\})\subseteq\hat\alpha$.
In particular,
$\eta_\beta(\{0,1\}\times\{-1,1\})\subseteq\hat\alpha $.
Hence the boundary components of $S$ that intersect with the good strip
$\eta_\beta$
intersects with $\hat\alpha $.

Now we denote the boundary components of $S$ by 
$\partial_1 S,\partial_2 S,\ldots,\partial_m S$.

\begin{claim2}\label{ccl2}
If $\partial_i S$ is good, then so is 
any strip intersecting with $\partial_i S$.
\end{claim2}

Suppose $\partial_i S$ is good.
$\hat\alpha\cap\partial_i S$ is a union of intervals on $\partial_i S$, 
and the endpoints of any of those intervals are contained in good strips.
Assume $\partial_i S$ also intersects with a strip that is not good. 
On $\partial_i S$,
one can choose a {\em nearest} pair of a good strip $\eta_{\beta_1}$ and a strip $\eta_{\beta_2}$ that is not good, for some $\beta_1,\beta_2\in\AAA$. 
This implies, there exists a closed interval $u$ on $\partial_i S$ 
such that $u$ intersects with $\eta_{\beta_1}$ and $\eta_{\beta_2}$, 
but not with any other strips. 
Since $\eta_{\beta_1}$ is good, the unique induced simple closed curve of $\beta_1$
is $\hat\alpha$, and so, $u\subseteq\hat\alpha$.
Since $u$ intersects 
with the induced simple closed curve of $\beta_2$,
we have a contradiction to the assumption that $\eta_{\beta_2}$ is not good.

\begin{figure}[htb!] 
\labellist 
\small\hair 2pt 
\pinlabel {$\scriptstyle\partial_i S$}  at 48 32
\pinlabel {$\scriptstyle u$} [b] at 46 50
\pinlabel {$\scriptstyle \eta_{\beta_1}$} at 10 52
\pinlabel {$\scriptstyle \eta_{\beta_2}$} [rb] at 91 64
\pinlabel {$\scriptstyle {\beta_1}$} [r] at 1 38
\pinlabel {$\scriptstyle {\beta_2}$} [l] at 106 61
\endlabellist 
\centering
\includegraphics{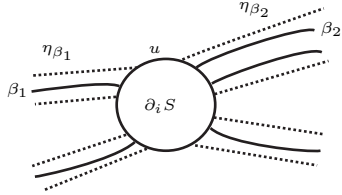}
\caption[Proof of Claim~\ref{ccl2} in Lemma~\ref{lem:1sided}]{Proof of Claim\ref{ccl2} in Lemma~\ref{lem:1sided}. \label{fig:good}}
\end{figure}

By Claim~\ref{ccl1} and~\ref{ccl2},
if $\partial_i S$ and $\partial_j S$ 
are connected by an arc in $\AAA$,
and $\partial_i S$ is good, then so is $\partial_j S$.
Since  $\ch(\alpha)$ is connected,
it follows that a boundary component of $S$ or a strip is good if and only if
it is contained in $\ch(\alpha)$.

Now choose any component $\kappa$ of $\partial  N \setminus\partial S$.
There exists $\beta\in\AAA$, such that $\kappa$ and the induced simple closed curve
$\hat\beta$ of $\beta$   bound an annulus contained in the closure of ${  N \setminus\ch(\alpha)}$,
and $\hat\beta\subseteq\ch(\alpha)$
 (Remark~\ref{rem:induced} (2)).
The strip $\eta_\beta$ of $\beta$ intersects with $\hat\beta$.
Since $\hat\beta\subseteq\ch(\alpha)$, 
$\eta_\beta\cap\ch(\alpha)\ne\varnothing$.
This implies that $\eta_\beta\subseteq \ch(\alpha)$, and so,  $\eta_\beta$ is good; that is,
 $\eta_\beta$ intersects also with $\hat\alpha$.
Since $\beta$ is assumed to be one-sided, 
$\hat\beta  = \hat\alpha$. 
Hence $\kappa$ is the unique boundary component of $N$ 
that bounds an annulus with $\hat\alpha$. 
This proves that $\partial  N$ contains only one component that is not
in $\partial S$.

Note that $N$ is a proper subsurface of $S$ such that the frontier, namely $\kappa$,
is connected.
Hence $\kappa$ separates $S$.

(2)
Suppose $\gamma$ is a curve on $S$, not intersecting with 
$(\cup\AAA)\cup\partial^* S$. 
Let $\AAA_0$ be a minimal set of arcs in $\AAA$ satisfying the following:
\begin{quote}
for each $\beta\in\AAA$, there uniquely exists $\beta_0\in\AAA_0$ such that
$\beta\sim\beta_0$.
\end{quote}

{\em Case 1. $\gamma$ is a closed curve.}

We have
\[
\gamma\subseteq S\setminus(\cup\AAA)
\subseteq S\setminus(\cup\AAA_0)
\sim
S\setminus(\cup_{\beta\in\AAA}\; \eta_\beta)
\subseteq
\overline{S\setminus\ch(\alpha)}
\sim
\overline{S\setminus  N }.
\]
The first homotopy is obtained by enlarging each arc in $\AAA_0$ to a strip,
and the second one is a deformation retract of the annuli
 discussed in Remark~\ref{rem:induced} (2) onto circles. 
Clearly, $\overline{S\setminus N}$
can be homotoped into $S\setminus N$.

{\em Case 2. $\gamma$ is a properly embedded arc.}

The argument for this case is almost the same as Case 1. 
One has only to show that
 that there exists a
homotopy
that sends
$\gamma$ into $S\setminus  N $,
leaving the endpoints on  $\partial S$.
For this, 
we choose $ N $ as a sufficiently small regular neighborhood of $\ch(\alpha)$
such that $\partial \gamma\cap  N =\varnothing$. This is possible
since $\gamma$ does not intersect $\partial^* S$.
Then we have only to note that
the homotopies in the proof of Case 1 do not move $\partial \gamma$.

(3)
Supose $\hat\alpha \sim 0$.
From (1),
there exists $\kappa\sim\hat\alpha $
 such that
$\partial  N \subseteq \{\kappa\}\cup\partial^* S$.
Since $\kappa$ separates, one can write $S =  N \cup S'$
 such that
 $ N \cap S' = \kappa$.
 $N$ contains at least one boundary component of $S$, namely 
any of the boundary components that $\alpha$ intersects. Hence, $N$ is not a disk.
Now for $\kappa$ to be null-homotopic, $S'$ must be a disk
and $\partial S\subseteq \partial  N $.
So $\partial  N  =\{\kappa\}\cup \partial S$,
and $\partial^* S = \partial S$.

Let $\gamma$ be any closed curve, not intersecting
with any arc in $\AAA$. By (2),
$ \gamma$ is homotopic into $S\setminus  N$, which is the interior of $S'$.
 This implies that $\gamma$ is null-homotopic.  \end{proof}

\section{Adding Bisimplicial Edges}\label{sec:bisimp}

An edge $\{a,b\}$ is {\em bisimplicial} if any vertex in $\link(a)$ is either equal or adjacent to 
 any vertex in $\link(b)$. For an edge $e$ of a graph, $\mathring{e}$ 
 denotes the interior of $e$. 
 In this section, we prove the following theorem.

\begin{thm}\label{thm:bisimplicial}
Let $e$ be a bisimplicial edge of a graph $\Gamma$.
If  $\Gamma\setminus\mathring{e}\in\NNN'$, then $\Gamma\in\NNN'$.
\end{thm}
 
\begin{proof} 
Write $e=\{a,b\}$,
and let $\Gamma' =  \Gamma\setminus\mathring{e}$. 
Assume that $\Gamma\not\in\NNN'$. One can find a compact hyperbolic surface 
$S$, and a relative embedding $\phi:\pi_1(S)\rightarrow A(\Gamma)$ with respect to a normalized label-reading pair
 $(\HH,\lambda)$.

First, consider the case when  $\lambda(\HH)\subseteq \link(a)\cup\{a\}$.
 $\Gamma_{\link(a)\cup \{a\}}\not\in\NNN'$, since
 the image of $\phi$ is in $A(\Gamma_{\link(a)\cup\{a\}})$.
$\Gamma_{\link(a)\cup \{a\}}$ is the join of the single vertex $a$ and $\Gamma_{\link(a)}$.
Since a single vertex is in $\NNN'$, 
Proposition~\ref{prop:join} implies that $\Gamma_{\link(a)}\not\in\NNN'$. 
Note that
$\Gamma_{\link(a)}\le\Gamma_{V(\Gamma)\setminus\{a\}}\le\Gamma'$.
Hence,
$\Gamma' $ is not in $\NNN'$, which contradicts to the assumption.
The case when
$\lambda(\HH)\subseteq \link(b)\cup\{b\}$ is similar.

Now assume 
  $\lambda(\HH)\not\subseteq \link(a)\cup\{a\}$ 
and
  $\lambda(\HH)\not\subseteq \link(b)\cup\{b\}$. We denote the boundary components of $S$ by
  $\partial_1 S,\partial_2 S,\ldots,\partial_m S$. For a based curve or an arc $\gamma$ on $S$,
we let  $w_\gamma$ denote the label-reading of $\gamma$ with respect to the label-reading pair
  $(\HH,\lambda)$ as in Section 2. 
  
\begin{claim}\label{cl1}
Suppose $\alpha$ and $\beta$ are essential simple closed curves on $S$ such that
$\alpha\cap\beta\ne\varnothing$, 
$w_\alpha\in\langle \link(a)\rangle$ and $w_\beta\in\langle \link(b)\rangle$.
Then
$\alpha \sim \beta^{\pm1}$.
\end{claim}
    
We may choose the base point of $\pi_1(S)$ in $\alpha\cap\beta$.
$\phi[[\alpha],[\beta]] = [w_\alpha, w_\beta] = 1$, 
since any vertex in $\link(a)$ is equal or adjacent to any vertex in $\link(b)$.
By Lemma~\ref{lem:primitive}, $\alpha\sim\beta^{\pm1}$. Here, we have assumed that $\alpha$
and $\beta$ are transverse to $\HH$. If not, one may consider $\alpha'\sim\alpha$ and 
$\beta'\sim\beta$ such that $\alpha'$ and $\beta'$ are sufficiently close to $\alpha$ and $\beta$
respectively, and transversely intersecting $\HH$. The claim is proved.
  
For $v\in V(\Gamma)$, recall that simple closed curves and properly embedded arcs
in $\HH$ labeled by $v$, are called $v$--curves and $v$--arcs, respectively.
$\CCC_v$ and $\AAA_v$ will denote the set of $v$--curves and the set of $v$--arcs,
respectively.
Let $\partial^v S$ denote the union of the boundary components of $S$ that intersect with 
 $v$--arcs.

\begin{claim}\label{cl2}
$(\cup\CCC_a)\cap(\cup \CCC_b)=\varnothing$.
\end{claim}
 
Suppose $\alpha$ and $\beta$ intersect at a point $p$, 
for some $\alpha\in\CCC_a$ and $\beta\in\CCC_b$.
One can find simple closed curves
$\alpha_1\sim\alpha$ and $\beta_1\sim\beta$, intersecting at a point $p'$ near $p$,
such that $\alpha_1$ and $\beta_1$ are transverse to $\HH$.
By requiring that $\alpha_1$ is sufficiently close to $\alpha$,
we may assume that the label-reading of $\alpha_1$ with the base point $p'$,
is same as the label-reading of $\alpha$  with a suitable choice of the base point.
If $\gamma\in\HH$ intersects with $\alpha$, then
$\lambda(\gamma)\in\link(a)$, 
by the definition of a label-reading pair.
Hence, $w_{\alpha_1}\in\langle \link(a)\rangle$.
Similarly, $w_{\beta_1}\in\langle \link(b)\rangle$. 
By Claim~\ref{cl1}, $\alpha\sim\alpha_1\sim\beta_1^{\pm1}\sim\beta^{\pm1}$,
which contradicts to the assumption that curves and arcs in $\HH$ are minimally intersecting
(Remark~\ref{rem:reduction} (1)).

\begin{claim}\label{cl3}
If $\alpha\in\AAA_a$ and $\beta\in\AAA_b$, then
$\alpha\not\sim\beta$.
\end{claim}
 Suppose an $a$--arc $\alpha$ and  a $b$--arc $\beta$ are homotopic.
 They join the same pair of boundary components, say $\partial_1 S$ and $\partial_2 S$.
$\alpha\sim\beta$ implies that if $\gamma\in\HH$ intersects with $\alpha$,
then $\gamma$ also intersects with $\beta$, and so, $\lambda(\gamma)\in\link(a)\cap\link(b)$.
It follows that $w_{\alpha}\in\langle  \link(a)\cap\link(b)\rangle$.
Note that $w_{\partial_1 S}$ and $w_{\partial_2 S}$ are in $\langle  a,b,\link(a)\cap\link(b)\rangle$.
As in the proof of Lemma~\ref{lem:simplicialv}, consider 
$\delta_1\sim\partial_1 S$ and $\delta_2\sim\partial_2 S$ with the same base point 
such that $\delta_1$ and $\delta_2$ transversely intersect $\HH$
(Figure~\ref{fig:lem:simplicialv}).
We may assume
$\delta_1$ and $\delta_2$ are sufficiently close to $\partial_1 S$ and $\alpha\cdot\partial_2 S\cdot\alpha^{-1}$
respectively,  so that $w_{\delta_1}$ and $w_{\delta_2}$ are in $\langle  a,b,\link(a)\cap\link(b)\rangle$.
Hence, $\phi([\delta_1,\delta_2]) = [w_{\delta_1}, w_{\delta_2}]=1$,
since $\{a,b\}\cup(\link(a)\cap\link(b))$ induces a complete subgraph in $\Gamma$.
This leads to a  contradiction again, 
since $\partial_1 S\ne\partial_2 S$ (Lemma~\ref{lem:norm1}) implies that $[\delta_1,\delta_2]\ne1$ (Lemma~\ref{lem:primitive}).
\begin{claim}\label{cl4}
If a $b$--arc $\beta$ joins two components in $\partial^a S$, 
then $\beta$ intersects some $\gamma\in\HH$ that is not
labeled by a vertex in $\link(a)\cup\{a\}$.
\end{claim}
As in the proof of Claim~\ref{cl3}, 
choose $\beta_1$ sufficiently close to $\beta$ such that $\beta_1\sim\beta$, 
$w_{\beta_1} = w_\beta\in\langle \link(b)\rangle$,
and $\beta_1$ transversely intersects $\HH$.
Assume that whenever $\gamma\in\HH$ and $\beta_1\cap\gamma\ne\varnothing$,
$\lambda(\gamma)\in\link(a)\cup\{a\}$. 
This implies $w_{\beta_1}\in\langle (\link(a)\cup\{a\})\cap\link(b)\rangle
= \langle  a,\link(a)\cap\link(b)\rangle$.
Let $\partial_1 S$ and $\partial_2 S$ be
the boundary components joined by $\beta$.
$\partial_i S\subseteq\partial^a S\cap\partial^b S$ for $i=1,2$, by the assumption of the claim.
This means $w_{\partial_1 S},w_{\partial_2 S} \in\langle  a,b,\link(a)\cap\link(b)
\rangle$.
As in the proof of Claim~\ref{cl3}, $[[\partial_1 S],[\beta_1\cdot\partial_2 S\cdot\beta_1^{-1}]]=1$,
which is a contradiction. So, there exists $\gamma\in\HH$  such that  
$\beta_1\cap\gamma\ne\varnothing$ (hence, $\beta\cap\gamma\ne\varnothing$)
and
$\lambda(\gamma)\not\in\link(a)\cup\{a\}$.

Now, we recall the notations and the terms from Definition~\ref{defn:stripchan} and Remark~\ref{rem:induced}. 
For each $\alpha\in\AAA_a$,
$\ch(\alpha)$ denotes the channel of $\alpha$ with respect to the set $\AAA_a$ and
 $N(\ch(\alpha))$ denotes a sufficiently small closed regular neighborhood of $\ch(\alpha)$ 
 satisfying the conclusion of Lemma~\ref{lem:1sided}.
This implies that $\hat\alpha$ is homotopic to the unique component of 
 $\partial N(\ch(\alpha))\setminus\partial S$ (Lemma~\ref{lem:1sided} (1)).

\begin{claim}\label{cl5}
Let $\alpha$ be an $a$--arc, and $\hat\alpha$ be the unique induced simple closed curve of $\alpha$
 with respect to $\AAA_a$. Then $w_{\hat\alpha}\in\langle \link(a)\rangle$.
\end{claim}

From Remark~\ref{rem:induced} that $\hat\alpha$ can be written as
\[ \hat\alpha = \alpha_1' \cdot \delta_1 \cdot\alpha_2' \cdot \delta_2 \cdots 
\alpha_r'\]
where for each $i$,
 $\alpha_i'$ is homotopic to an $a$--arc $\alpha_i$, 
 and $\delta_i$ is an interval on a boundary component of $S$ 
 which is intersecting with the $a$--arcs $\alpha_i$ and $\alpha_{i+1}$. 
 Moreover, $\hat\alpha$ does not intersect with
any $a$--curves or $a$--arcs. It follows that $w_{\hat\alpha}\in\langle \link(a)\rangle$.

\begin{claim}\label{cl6}
The induced simple closed curve of an $a$-- or $b$--arc is essential.
\end{claim}

Suppose the induced simple closed curve $\hat\alpha$
of an $a$--arc $\alpha$ is null-homotopic. By Lemma~\ref{lem:1sided} (3), 
$\partial^a S = \partial S$, and any simple closed curve in $\HH$  is labeled by a vertex in $\link(a)$. 
This implies that
the label of any curve or arc in $\HH$ is either $a$ or  adjacent to $a$.
Hence, $\lambda(\HH)\in\link(a)\cup\{a\}$, which was excluded. 
The case for the induced simple closed curve of a $b$--arc
is similar, by symmetry.

\begin{claim}\label{cl7}
$(\cup\AAA_a)\cap(\cup\CCC_b)=\varnothing$, and 
$(\cup\AAA_b)\cap(\cup\CCC_a)=\varnothing$.
\end{claim}

Supose $\alpha\in\AAA_a$ and $\beta\in\CCC_b$ intersect at $p$.
Let $\hat\alpha$ be the induced simple closed curve of $\alpha$.
 $\hat\alpha'$ denotes the unique
boundary component of $ N(\ch(\alpha)) $ that is not in $\partial S$
(Lemma~\ref{lem:1sided} (1)). $\hat\alpha'\sim\hat\alpha$, and
 $\hat\alpha'\not\sim 0$ by Claim~\ref{cl6}.
By Claim~\ref{cl5},
$w_{\hat\alpha'}\in\langle \link(a)\rangle$.
Since $\alpha\cap\beta\ne\varnothing$,
$\hat\alpha'\cap\beta\ne\varnothing$. 
Moreover, $w_\beta\in\langle\link(b)\rangle$.
By Claim~\ref{cl1}, $\hat\alpha'\sim\beta^{\pm1}$, and so, $i(\alpha,\beta) = i(\alpha,\hat\alpha') = 0$.
This contradicts to the assumption that 
$\alpha$ and $\beta$ are minimally intersecting. 
$(\cup\AAA_b)\cap(\cup\CCC_a)=\varnothing$ follows from the symmetry.

\begin{claim}\label{cl8}
Let $\alpha\in \AAA_a$ and $\beta\in\AAA_b$.
Denote the induced simple closed curves of $\alpha$ and $\beta$
by $\hat\alpha$ and $\hat\beta$, respectively. 
Suppose either
\enumir
\be
\item
$\alpha$ and $\beta$ intersect, or
\item
$\hat\alpha$ and $\hat\beta$ intersect, 
and there exists a boundary component which intersects with both $\alpha$ and $\beta$.
\end{enumerate}
\enumia
Then 
$\hat\alpha\sim\hat\beta^{\pm1}$,
and
$ N(\ch(\alpha)) \cap\partial S =  N(\ch(\beta)) \cap\partial S$.
Moreover, there exists a homotopy from 
$ N(\ch(\alpha)) $ onto $ N(\ch(\beta)) $ fixing $ N(\ch(\alpha)) \cap\partial S
=  N(\ch(\beta)) \cap\partial S$.
\end{claim}

As in Claim~\ref{cl7}, let $\hat\alpha'$ and $\hat\beta'$
be the boundary components of $ N(\ch(\alpha)) $ and $ N(\ch(\beta))$,
that are not boundary components of $S$, respectively.
Assuming (i) or (ii),
two essential curves $\hat\alpha'$ and $\hat\beta'$ intersect.
Here, we have also assumed that 
$\hat\alpha'$ and $\hat\beta'$ are sufficiently close to $\hat\alpha$ and $\hat\beta$.
By Claim~\ref{cl5},
$w_{\hat\alpha'}\in\langle \link(a)\rangle$ and  $w_{\hat\beta'}\in\langle \link(b)\rangle$.
From Claim~\ref{cl1}, it follows that
 $\hat\alpha'\sim\hat\beta'^{\pm1}$.
By Lemma~\ref{lem:1sided} (1) ,
both $\hat\alpha'$ and $\hat\beta'$ are separating simple closed curves on $S$.
So either $ N(\ch(\alpha)) \sim  N(\ch(\beta)) $
or $ N(\ch(\alpha)) \sim \overline{S\setminus  N(\ch(\beta)) }$.
Suppose $ N(\ch(\alpha)) \sim \overline{S\setminus  N(\ch(\beta)) }$.
Then any boundary component of $S$ contained in $ N(\ch(\alpha)) $ will not be contained in 
$ N(\ch(\beta)) $.
So no boundary component of $S$ can intersect both $\alpha$ and $\beta$.
Since $\alpha\subseteq N(\ch(\alpha))$, $\alpha$ is homotopic into
$S\setminus  N(\ch(\beta)) $, 
and so, $i(\alpha,\beta)=0$. So neither (i) nor (ii) of the given 
conditions holds. 
Hence $ N(\ch(\alpha)) \sim  N(\ch(\beta)) $, and the rest of the claim follows immediately.

\begin{claim}\label{cl9}
Let $\alpha$ be an $a$--arc. 
Suppose a $b$-arc $\beta$  joins two boundary components of $S$
 that are contained in $N(\ch(\alpha))$.
Let  $\hat\alpha$ and $\hat\beta$ denote the induced simple closed curves of
$\alpha$ and $\beta$, respectively. 
Then $\hat\alpha\cap\hat\beta=\varnothing$.
\end{claim}

Suppose $\beta$ joins $\partial_1 S$ and $\partial_2 S$,
and $\partial_1 S\cup\partial_2 S\subseteq N(\ch(\alpha)) $.
Assume that $\hat\alpha\cap\hat\beta\ne\varnothing$.
For $i=1$ or $2$,  $\partial_i S\subseteq N(\ch(\alpha))$ and so,
$\partial_i S$ intersects with some $a$--arc.
By Claim~\ref{cl4}, there exists  $\gamma\in\HH$ such that
$\beta\cap\gamma\ne\varnothing$ and 
  $\lambda(\gamma)\not\in\link(a)\cup\{a\}$.
This implies that 
 $\gamma$ can not intersect any $a$-curve or $a$-arc, and
 $\gamma\cap\partial^a S=\varnothing$.
 In particular, $\gamma\cap (\cup \AAA_a)=\varnothing$.
By Lemma~\ref{lem:1sided} (2),
$\gamma\leadsto S\setminus  N(\ch(\alpha))$ (Notation~\ref{notation:hom}).
By choosing a suitable arc in $\ch(\alpha)$, 
we may assume that $\alpha$ intersects with
either $\partial_1 S$ or $\partial_2 S$.
From Claim~\ref{cl8} (with condition (ii)),
$\gamma\leadsto S\setminus  N(\ch(\beta)) $.
So $i(\beta,\gamma)=0$, which is a contradiction.

\begin{claim}\label{cl10}
$(\cup\AAA_a)\cap(\cup\AAA_b)=\varnothing$.
\end{claim}

Suppose $\alpha\in\AAA_a$ and $\beta\in\AAA_b$ intersect.
By Claim~\ref{cl8} (with condition (i)),
 $\hat\alpha\sim\hat\beta^{\pm1}$ and $ N(\ch(\alpha)) \sim  N(\ch(\beta)) $.
This implies that 
$  N(\ch(\alpha)) \cap\partial S =  N(\ch(\beta)) \cap\partial S$,
and so,
$\beta\subseteq  N(\ch(\beta)) $ joins two boundary components contained 
in $ N(\ch(\alpha)) $. By Claim~\ref{cl9}, $\hat\alpha\cap\hat\beta=\varnothing$,
which contradicts to the assumption that $\alpha\cap\beta\ne\varnothing$.

\begin{claim}\label{cl11}
$\partial^a S\cap \partial^b S= \varnothing$.
\end{claim}

Suppose $\partial_i S$ intersects with an $a$--arc $\alpha$
and a $b$--arc $\beta$. By considering  a nearest pair of such arcs
 on $\partial_i S$,
we may assume that
the induced simple closed curves $\hat\alpha$
and $\hat\beta$ of $\alpha$ and $\beta$ intersect (Figure~\ref{fig:claim11}).
By Claim~\ref{cl8} again, 
$N(\ch(\alpha))=N(\ch(\beta))$, and hence as in the proof of Claim~\ref{cl10},
$\beta$ joins two boundary components of $ N(\ch(\alpha)) $.
By Claim~\ref{cl9}, $\hat\alpha$ and $\hat\beta$ are disjoint, which is a contradiction.
This proves Claim~\ref{cl11}. 

 \begin{figure}[bh!] 
 \labellist 
\small\hair 2pt 
\pinlabel {$\scriptstyle \hat\alpha\cap\hat\beta$} [b] at 26 53
\pinlabel {$\bullet$} at 28 46
\pinlabel {$\scriptstyle\partial_i S$}  at 66 28
\pinlabel {$\scriptstyle\hat\alpha$} [tr] at 27 10
\pinlabel {$\scriptstyle\hat\beta$} [l] at 98 10
\pinlabel {$\scriptstyle\alpha$} [l] at 135 47
\pinlabel {$\scriptstyle\beta$} [rb] at  -5 32
\pinlabel {$\swarrow$} [tr] at 23 40
\pinlabel {$\scriptstyle b$} [tr] at 13 27
\pinlabel {$\scriptstyle a$} [tl] at 123 35
\pinlabel {$\searrow$} [tl] at 112 48
\endlabellist 
\centering
\includegraphics{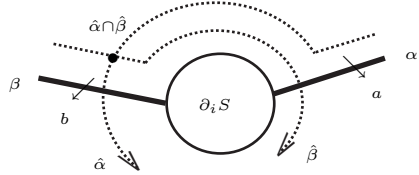}
\caption{Proof of Claim~\ref{cl11} in Theorem~\ref{thm:bisimplicial}. }
\label{fig:claim11}
\end{figure}

Recall 
 $\Gamma' =  \Gamma\setminus\mathring{e}$. 
By Claim~\ref{cl2}, \ref{cl7} and \ref{cl10}, 
$\alpha$ and $\beta$ are disjoint
for any $\alpha\in\lambda^{-1}(a)$
and
$\beta\in\lambda^{-1}(b)$.
Hence, $(\HH,\lambda)$ can be considered as a label-reading pair
with the underlying graph $\Gamma'$, inducing
$\phi':\pi_1(S)\rightarrow A(\Gamma')$. 
Injectivity of $\phi'$ can be seen from the following commutative diagram.
Here, $A(\Gamma')\rightarrow A(\Gamma)$ is the natural
quotient map, obtained by adding the relator $[a,b]=1$.
\[
\xymatrix{
 & A(\Gamma')\ar@{>>}[d]^*{[a,b]\mapsto 1}\\
 \pi_1(S)\ar@{.>}[ur]^{\phi'}\ar[r]^\phi &
 A(\Gamma) 
 }
 \]
By Claim~\ref{cl11}, no boundary components of $S$ intersect with an $a$--arc and a $b$--arc at the same time. 
So the labels of the arcs intersecting with a fixed boundary component $\partial_i S$
are pairwise adjacent not only in $\Gamma$, but also in $\Gamma'$.
Hence $\phi'$ is a relative embedding, and $\Gamma'\not\in\NNN'$. 
  \end{proof}

Recall that a graph $\Gamma$ is {\em chordal bipartite} if $\Gamma$ does not contain a triangle
or an induced cycle of length at least $5$.

\begin{cor}\label{cor:cb}
Choral bipartite graphs are in $\NNN'$.
\end{cor}

\begin{proof}
By applying Theorem~\ref{thm:completeamalgam} to  $K_0$ ($=\varnothing$) amalgamation,
one sees that $\NNN'$ is closed under disjoint union. 
In particular, discrete graphs are in $\NNN'$.
Golumbic and Goss proved that 
by removing bisimplicial edges from a chordal bipartite graph successively,
one obtains a discrete graph~\cite{GG1978}. By Theorem~\ref{thm:bisimplicial},
it follows that any chordal bipartite graph is in $\NNN'$. \end{proof}

\section{$\NNN$ and $\NNN'$}\label{sec:nni}

We have shown several properties of the graph class $\NNN'$, which is a subclass of $\NNN$.
In this section, we give a lower bound for $\NNN'$ to illustrate that $\NNN'$ already contains a large number of
graphs. Also, we prove that two specific graphs are not in $\NNN$ (hence not in $\NNN'$), 
providing new examples not covered by the results that we have discussed so far. Finally, we show
equivalent formulations of Conjecture~\ref{conj:nclosedkn}.

Let $\Gamma$ be a graph. Suppose $B$ is a subset of $V(\Gamma)$ such that
the complement graph of the induced subgraph $\Gamma_B$ is connected. 
Recall from~\cite{kim2008} that 
the {\em co-contraction $\cob(\Gamma,B)$ of $\Gamma$ relative to $B$} is defined as
\[
\cob(\Gamma,B) = \overline{ \overline{\Gamma} \big/ \overline{\Gamma_B}},\]
where $\overline{\Gamma}\big/\overline{\Gamma_B}$ denote the graph obtained from $\overline{\Gamma}$
by topologically contracting each edge in $\overline{\Gamma_B}$ onto a vertex
and removing loops or multi-edges thus obtained, successively. 
In~\cite{kim2008}, it is shown that $A(\cob(\Gamma,B))$ embeds
into $A(\Gamma)$. Using this, we first prove that $\NNN'$ is closed under co-contraction.

\begin{prop}\label{prop:cocontni}
Let $\Gamma$ be a graph that co-contracts onto $\Gamma'$. If $\Gamma\in\NNN'$, then $\Gamma'\in\NNN'$.
\end{prop}
\begin{proof}
Assume $\Gamma'\not\in\NNN'$. 
One can find a compact hyperbolic surface $S$ and a relative embedding
 $\phi:\pi_1(S)\rightarrow A(\Gamma')$ 
 induced by a normalized label-reading pair $(\HH',\lambda')$. 
 Let $\partial_1 S,\ldots,\partial_m S$ be the boundary components of $S$.
Using induction, we have only to consider the case when $\overline{\Gamma'}$ is obtained from $\overline{\Gamma}$
by contracting an edge $\{a,b\}\in E(\overline{\Gamma})$ onto a vertex $v\in V(\overline{\Gamma'})$.
From~\cite{kim2008}, the map $A(\Gamma')\rightarrow A(\Gamma)$ sending $v$ to $b^{-1}ab$, while sending
the other vertices onto themselves, is an embedding.
From the definition of a co-contraction, $\link(v) = \link(a)\cap\link(b)$.
If $\lambda'^{-1}(v)=\varnothing$, then $\phi:\pi_1(S)\rightarrow A(\Gamma'\setminus\{v\})=A(\Gamma\setminus\{a,b\})
\le A(\Gamma)$, and hence $\Gamma\not\in\NNN'$.
Now assume $\lambda'^{-1}(v)\ne\varnothing$ and choose any $\alpha\in\lambda'^{-1}(v)$. First, consider the case
when $\alpha$ is an arc.
Suppose
$\alpha$ intersects with $\partial_i S$.
By the definition of a relative embedding, any arc intersecting with $\partial_i S$ is labeled by a vertex in 
$\{v\}\cup \link(v) = \{v\}\cup (\link(a)\cap\link(b))$. 
Consider the strip $\eta_\alpha:I\times[-1,1]\rightarrow S$ containing $\alpha$. We replace $\alpha$ in $\HH'$ by 
three homotopic arcs $\alpha_1=\eta_\alpha(I\times{-1}), \alpha_2=\eta_\alpha(I\times{0})$ and 
$\alpha_3=\eta_\alpha(I\times{1})$
such that the following hold (Figure~\ref{fig:cocont}\;(a)).
\enumir
\be
\item
$\alpha_1$ and $\alpha_3$ are labeled by $b$.
\item
$\alpha_2$ is labeled by $a$.
\item
$\alpha_2$ and $\alpha_3$ have the transverse orientations induced by the homotopies  $\alpha_2\sim\alpha$
and $\alpha_3\sim\alpha$, 
while $\alpha_1$ has the opposite orientations.
\ee
\enumia
Apply this process for each $v$--arc,
and also similarly for each $v$--curve.
This results in a new label-reading pair, denoted by $(\HH,\lambda)$, 
with the underlying graph $\Gamma$. The induced label-reading map 
$\psi:\pi_1(S)\rightarrow A(\Gamma)$ 
is the composition of embeddings
$\pi_1(S)\hookrightarrow A(\Gamma')\hookrightarrow A(\Gamma)$.

In $(\HH,\lambda)$, suppose $\partial_i S$ intersects with a $b$--arc, for some $i$.
Any arcs intersecting with $\partial_i S$ are labeled by either $a$, $b$ or vertices in $\link(a)\cap\link(b)$.
One can pair $b$--arcs intersecting with $\partial_i S$, such that for each pair $\{\beta_1,\beta_2\}$:
\enumir
\be
\item
the transverse orientations of $\beta_1$ and $\beta_2$ are opposite to each other,
\item
one of the intervals on $\partial_i S\setminus(\beta_1\cup\beta_2)$ intersects only with
the arcs labeled by $\link(a)\cap\link(b)$.
\ee
Then one can remove intersections between $\partial_i S$ and $b$--arcs,
without altering the equivalence class of $(\HH,\lambda)$ as is illustrated in Figure~\ref{fig:cocont}\;(b).
By applying this process to any $\partial_i S$ intersecting with a $b$--arc, 
we obtain another label-reading pair $(\HH_1,\lambda_1)$
such that the labels of the arcs intersecting with each boundary component
induce a complete subgraph of $\Gamma$.
Hence, $\psi$ is a relative embedding, and $\Gamma\not\in\NNN'$.
 \end{proof}

\begin{figure}[htb!] 
\centering
\subfigure[]{
\labellist 
\small\hair 2pt 
\pinlabel {$\scriptstyle \partial_i S$} [lb] at 11 32
\pinlabel {$\scriptstyle (\HH',\lambda')$} [b] at 14 -4
\pinlabel {$\scriptstyle \partial_i S$} [lb] at 93 32
\pinlabel {$\scriptstyle (\HH,\lambda)$} [b] at 97 -4
\pinlabel {$\downarrow$} [t] at 39 41
\pinlabel {$\leadsto$} [l] at 61 25
\pinlabel {$\downarrow$} [t] at 123 43
\pinlabel {$\downarrow$} [t] at 123 28
\pinlabel {$\uparrow$} [t] at 120 58
\pinlabel {${}_\alpha$} [l] at 52 37
\pinlabel {${}_{\alpha_1}$} [l] at 134 51
\pinlabel {${}_{\alpha_2}$} [l] at  134 37
\pinlabel {${}_{\alpha_3}$} [l] at  134 24
\pinlabel {${}_v$} [t] at 40 28
\pinlabel {${}_b$} [l] at 121 57
\pinlabel {${}_a$} [l] at 123 44 
\pinlabel {${}_b$} [l] at 124 15
\endlabellist
\centering
\includegraphics{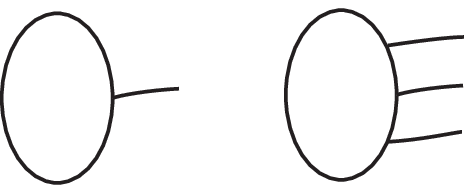}
}
\hfill
\subfigure[]{
\labellist 
\small\hair 2pt  
\pinlabel {$\scriptstyle\partial_i S$}  at 37 33
\pinlabel {$\scriptstyle (\HH,\lambda)$} [b] at 37 -4
\pinlabel {$\scriptstyle \partial_i S$}  at 146 37
\pinlabel {$\scriptstyle (\HH_1,\lambda_1)$} [b] at 146 -4
\pinlabel {$\downarrow$} [b] at 15 14
\pinlabel {$\uparrow$} [t] at 14 39
\pinlabel {$\uparrow$} [t] at 14 50
\pinlabel {${}_{\beta_1}$} [t] at -5 45
\pinlabel {${}_b$} [t] at 8 51
\pinlabel {${}_a$} [t] at 8 38
\pinlabel {${}_b$} [b] at 10 11
\pinlabel {$\leadsto$} [l] at 87 25
\pinlabel {$\leftarrow$} [l] at 23 60
\pinlabel {$\rightarrow$} [r] at 49 59
\pinlabel {${}_c$} [l] at 18 61
\pinlabel {${}_d$} [r] at 55 61
\pinlabel {$\uparrow$} [b] at 63 41
\pinlabel {$\downarrow$} [b] at 62 28
\pinlabel {$\downarrow$} [b] at 62 16
\pinlabel {${}_{\beta_2}$} [t] at 81 53
\pinlabel {${}_b$} [b] at 68 38
\pinlabel {${}_a$} [b] at 67 27
\pinlabel {${}_b$} [t] at 67 23
\pinlabel {$\downarrow$} [b] at 120 17
\pinlabel {$\uparrow$} [t] at 119 42
\pinlabel {$\uparrow$} [t] at 119 53
\pinlabel {${}_b$} [t] at 118 60
\pinlabel {${}_a$} [t] at 125 43
\pinlabel {${}_b$} [b] at 120 9
\pinlabel {$\leftarrow$} [l] at  131 65
\pinlabel {$\rightarrow$} [r] at 160 66
\pinlabel {${}_c$} [l] at 125 65
\pinlabel {${}_d$} [r] at 167 66
\pinlabel {$\downarrow$} [b] at 180 32
\pinlabel {${}_{\beta}$} [t] at 190 56
\pinlabel {${}_a$} [b] at 188 36
\endlabellist
\centering
\includegraphics{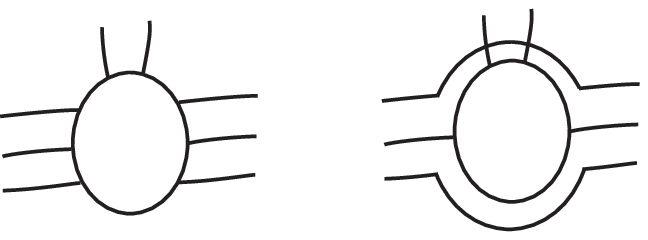}
}
\caption{Proof of Proposition~\ref{prop:cocontni}. In (b), $c$ and $d$ belong to $\link(a)\cap\link(b)$ in $A(\Gamma)$.
Hence, $c$--and $d$--arcs are allowed to intersect with the $b$--curve or $b$--arc $\beta$ 
in $(\HH_1,\lambda_1)$.}
\label{fig:cocont}
\end{figure}

Define $\FFF$ to be the smallest family of graphs satisfying the following conditions.
\enumir
\be
\item $K_n\in\FFF$.
\item $\Gamma_1,\Gamma_2\in\FFF$, then $\join(\Gamma_1,\Gamma_2)\in\FFF$.
\item If $\Gamma_1,\Gamma_2\in\FFF$, and $\Gamma$ is a complete graph amalgamation of $\Gamma_1$
and $\Gamma_2$, then $\Gamma\in\FFF$.
\item Suppose $e$ is a bisimplicial edge of a graph $\Gamma$. 
If $\Gamma\setminus\mathring{e}\in\FFF$, then $\Gamma\in\FFF$.
\item Let $\Gamma\in\FFF$ and $B\subseteq V(\Gamma)$  such that $\overline{\Gamma_B}$ is connected.
Then $\cob(\Gamma,B)\in\FFF$.
\end{enumerate}
\enumia

By the Dirac's result in~\cite{dirac1961} which was used in the proof of Corollary~\ref{cor:chordal},
 (i) and (iii) imply that chordal graphs are in $\FFF$. 
 The result of Golumbic and Goss~\cite{GG1978} quoted in the proof of Corollary~\ref{cor:cb}, along with
(iv), implies that any chordal bipartite graphs are in $\FFF$. 

\begin{cor}\label{cor:fff}
$\NNN'$ contains $\FFF$.
\end{cor}

\begin{proof}
Proposition~\ref{prop:join}, Proposition~\ref{prop:cocontni}, 
Theorem~\ref{thm:completeamalgam} and Theorem~\ref{thm:bisimplicial}
imply that $\NNN'$ is closed under taking a join, taking a co-contraction, amalgamating along a complete subgraph
and adding a bisimplicial edge,
respectively. Since $\FFF$ is the smallest of such a graph class, $\FFF\subseteq\NNN'$.
 \end{proof}

So, $\FFF$ provides a lower bound for $\NNN'$. 
As Corollary~\ref{cor:fff} summarizes techniques introduced in this paper,
it seems likely that determining whether $\FFF=\NNN'$ will require new insights.

Crisp, Sageev and Sapir proposed several {\em reduction moves} on underlying graphs of label-reading maps,
which they successfully used to classify all the graphs in $\NNN$ with at most 8 vertices~\cite{CSS2008}.
More precisely, they described eight {\em forbidden graphs}, and proved that 
a graph with at most 8 vertices is in $\NNN$ if and only if the graph does not contain
any forbidden graph as an induced subgraph.
Their beautiful arguments, especially of finding candidates for kernel elements of 
label-reading maps, also excluded many graphs 
with 9 or more vertices from $\NNN$. However, 
the question of classifying all the graphs on which right-angled Artin groups
 contain closed hyperbolic surface groups currently seems wide-open.
Here, we provide two new examples of graphs that are not in $\NNN$.

\begin{exmp}\label{exmp:nine}
Crisp, Sageev and Sapir proved that the right-angled Artin group 
on the graph $P_1(8)$ contains a closed hyperbolic surface group~\cite{CSS2008}.
The complement graph of $P_1(8)$ is drawn in Figure~\ref{fig:exception}\;(a).
Now consider the graphs $\Gamma_1$ and $\Gamma_2$, 
whose complements are drawn in Figure~\ref{fig:exception}\;(b) and (c), respectively.
If we topologically contract the edge $\{a,b\}$ in the complement of $\Gamma_1$,
and remove multi-edges thus obtained, then we have the complement  graph of $P_1(8)$.
This means $\Gamma_1$ co-contracts onto $P_1(8)$, and hence,
we have an embedding $A(P_1(8))\hookrightarrow A(\Gamma_1)$~\cite{kim2008}.
Similarly, $\overline{\Gamma_2}$ contracts onto $\overline{\Gamma_1}$ 
by contracting the edge $\{c,d\}$ onto $a$,
and so, $A(\Gamma_1)$ embeds into $A(\Gamma_2)$.
This implies that 
$A(\Gamma_1)$ and $A(\Gamma_2)$ contain closed hyperbolic surface groups,
since so does $A(P_1(8))$.
One can easily check that $\Gamma_1$ and $\Gamma_2$
do not contain any {\em forbidden subgraphs} considered  in~\cite{CSS2008}.
This gives new examples of graphs not in $\NNN$, hence not in $\NNN'$.
\end{exmp}

\begin{figure}[htb!]
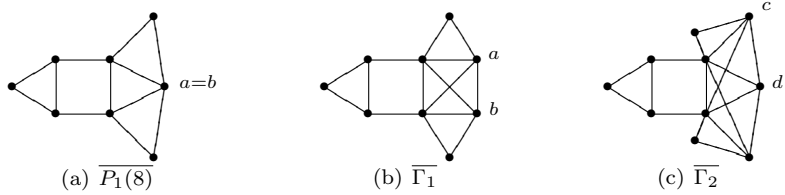
 
\centering
\subfigure[$\overline{P_1(8)}$]{
$$
\xy
0;/r.17pc/:
(-8,5)*{}="a";
(0,0)*{}="b"; 
(0,10)*{}="c"; 
(10,0)*{}="d"; 
(10,10)*{}="e"; 
(20,5)*{}="f"; 
(26,6)*{\scriptstyle a=b};
(18,-8)*{}="g";
(18,18)*{}="h";
"a"*{{}_\bullet};
"b"*{{}_\bullet};
"c"*{{}_\bullet};
"d"*{{}_\bullet};
"e"*{{}_\bullet};
"f"*{{}_\bullet};
"g"*{{}_\bullet};
"h"*{{}_\bullet};
"a";"b"**\dir{-};
"a";"c"**\dir{-};
"c";"b"**\dir{-};
"d";"b"**\dir{-};
"e";"c"**\dir{-};
"e";"d"**\dir{-};
"f";"d"**\dir{-};
"e";"f"**\dir{-};
"g";"f"**\dir{-};
"h";"f"**\dir{-};
"g";"d"**\dir{-};
"e";"h"**\dir{-};
\endxy
$$
}
\hspace{.3in}
\subfigure[$\overline{\Gamma_1}$]{
$$
\xy
0;/r.17pc/:
(-8,5)*{}="a";
(0,0)*{}="b"; 
(0,10)*{}="c"; 
(10,0)*{}="d"; 
(10,10)*{}="e"; 
(20,0)*{}="f"; 
(20,10)*{}="g";
(15,-8)*{}="h";
(15,18)*{}="i";
(23,1)*{\scriptstyle b};
(23,11)*{\scriptstyle a};
"a"*{{}_\bullet};
"b"*{{}_\bullet};
"c"*{{}_\bullet};
"d"*{{}_\bullet};
"e"*{{}_\bullet};
"f"*{{}_\bullet};
"g"*{{}_\bullet};
"h"*{{}_\bullet};
"i"*{{}_\bullet};
"a";"b"**\dir{-};
"a";"c"**\dir{-};
"c";"b"**\dir{-};
"d";"b"**\dir{-};
"e";"c"**\dir{-};
"e";"d"**\dir{-};
"f";"d"**\dir{-};
"e";"f"**\dir{-};
"g";"f"**\dir{-};
"h";"f"**\dir{-};
"g";"d"**\dir{-};
"h";"d"**\dir{-};
"e";"g"**\dir{-};
"i";"g"**\dir{-};
"e";"i"**\dir{-};
\endxy
$$
}
\hspace{0.3in}
\subfigure[$\overline{\Gamma_2}$]{
$$
\xy
0;/r.17pc/:
(-8,5)*{}="a";
(0,0)*{}="b"; 
(0,10)*{}="c"; 
(10,0)*{}="d"; 
(10,10)*{}="e"; 
(20,5)*{}="f"; 
(23,6)*{\scriptstyle d};
(18,-8)*{}="g";
(18,18)*{}="h";
(21,20)*{\scriptstyle c};
(8,-5)*{}="i";
(8,15)*{}="j";
"a"*{{}_\bullet};
"b"*{{}_\bullet};
"c"*{{}_\bullet};
"d"*{{}_\bullet};
"e"*{{}_\bullet};
"f"*{{}_\bullet};
"g"*{{}_\bullet};
"h"*{{}_\bullet};
"i"*{{}_\bullet};
"j"*{{}_\bullet};
"a";"b"**\dir{-};
"a";"c"**\dir{-};
"c";"b"**\dir{-};
"d";"b"**\dir{-};
"e";"c"**\dir{-};
"e";"d"**\dir{-};
"f";"d"**\dir{-};
"e";"f"**\dir{-};
"g";"f"**\dir{-};
"h";"f"**\dir{-};
"g";"d"**\dir{-};
"e";"h"**\dir{-};
"i";"g"**\dir{-};
"i";"d"**\dir{-};
"e";"j"**\dir{-};
"j";"h"**\dir{-};
"h";"d"**\dir{-};
"e";"g"**\dir{-};
\endxy
$$
}
\caption{The complement graphs of  $P_1(8),\Gamma_1$ and $\Gamma_2$.}
\label{fig:exception}
\end{figure}

One of the key obstructions for the question of classifying graphs in $\NNN$ is
Conjecture~\ref{conj:nclosedkn}. 
Note that $\NNN$ is closed under disjoint union and amalgamating along a vertex~\cite{kim2007,CSS2008}. 
We conclude this article by listing
equivalent formulations to Conjecture~\ref{conj:nclosedkn}.

\enumir
\begin{prop}\label{prop:nniequiv}
The following are equivalent.
\be
\item
$\NNN$ is closed under complete graph amalgamation.
\item
If $\Gamma'$ is obtained from $\Gamma$ by removing a simplicial vertex,
and $\Gamma'\in\NNN$, then $\Gamma\in\NNN$.
\item
$\NNN' = \NNN$.
\ee
\end{prop}
\enumia

\begin{proof}
(i)$\Rightarrow$(ii) is  Obvious, since adding a simplicial vertex to $\Gamma'$ 
is same as 
amalgamating $\Gamma'$ 
with a complete graph $K_n$ along $K_{n-1}$ for some $n$.

For (ii)$\Rightarrow$(iii), first note that $\NNN'\subseteq\NNN$ by definition. 
To prove $\NNN\subseteq\NNN'$, suppose $\Gamma\not\in\NNN'$.
There exists a relative embedding
$\phi:\pi_1(S)\rightarrow A(\Gamma)$, for some compact hyperbolic surface $S$.
By Lemma~\ref{lem:promotion}, $A(\Gamma^*)$ contains a closed hyperbolic
surface group. Note that
 $\Gamma^*$ is obtained from $\Gamma$ by adding 
independent simplicial vertices to $\Gamma$. Assuming (ii), $\Gamma\not\in\NNN$.

(iii)$\Rightarrow$ (i) is an immediate from Theorem~\ref{thm:completeamalgam}. \end{proof}

\end{document}